\documentclass[reqno,11pt,a4paper]{amsart}
\usepackage[margin=3cm]{geometry}                       
\usepackage{graphicx,caption,subcaption}      
\usepackage{amssymb,anyfontsize,dsfont,enumerate,fix-cm,lpic,mathtools,stmaryrd,tensor,thmtools,tikz,txfonts,xspace}
\usepackage[all]{xy}
\usepackage[oldstyle]{libertine}%
\usepackage[T1]{fontenc}
\usepackage[utf8]{inputenc}
\makeatletter
\renewcommand*\libertine@figurestyle{LF}
\makeatother
\usepackage[libertine,libaltvw,liby]{newtxmath}
\makeatletter
\renewcommand*\libertine@figurestyle{OsF}
\makeatother
\usepackage[style=alphabetic,giveninits,sorting=nyt,maxbibnames=99]{biblatex}
\usepackage[pdftex,colorlinks=true,linkcolor=blue,citecolor=purple,filecolor=magenta,%
urlcolor=cyan]{hyperref}
\usepackage[noabbrev]{cleveref}
\expandafter\def\csname ver@etex.sty\endcsname{3000/12/31}

\usepackage{autonum}

\theoremstyle{plain}
\newtheorem{theorem}{Theorem}[section]
\newtheorem{proposition}[theorem]{Proposition}
\newtheorem{lemma}[theorem]{Lemma}
\newtheorem{corollary}[theorem]{Corollary}

\theoremstyle{definition}

\newtheorem{definition}[theorem]{Definition}

\theoremstyle{remark}
\newtheorem{remark}[theorem]{Remark}

\DeclareMathOperator{\im}{Im}
\DeclareMathOperator{\Id}{Id}

\def\R{\mathbb{R}}	
\renewcommand{\leq}{\leqslant} 		
\renewcommand{\geq}{\geqslant}

\let\ker\relax
\DeclareMathOperator{\ker}{Ker}
\DeclareMathDelimiter\llbracket{\mathopen}{stmry}{"4A}
                                          {stmry}{"71}
\DeclareMathDelimiter\rrbracket{\mathclose}{stmry}{"4B}
                                           {stmry}{"79}

\newcommand{\mc}[1]{\mathcal{#1}}

\def\cA{\mathcal{A}}
\def\hcA{\hat{\cA}}

\def\hcC{\hat{\mathcal{C}}}
\def\cD{\mathcal{D}}
\def\cE{\mathcal{E}}
\def\cF{\mathcal{F}}
\def\hcF{\hat{\cF}}
\def\cM{\mathcal{M}}
\def\hcM{\hat{\cM}}
\def\m{\mathfrak{m}}
\def\plh{\mathord{\cdot}}

\def\l{\lambda}
\def\d{\partial}
\def\th{\theta}

\begin{document}

\title{Central invariants revisited}

\author{Guido Carlet}
\address{Institut de Math\'ematiques de Bourgogne, UMR 5584 CNRS, Universit\'e de Bourgogne Franche-Comt\'e, 21000 Dijon, France.
}
\email{guido.carlet@u-bourgogne.fr}

\author{Reinier Kramer}
\address{Korteweg-de Vries Instituut voor Wiskunde, 
Universiteit van Amsterdam, Postbus 94248,
1090GE Amsterdam, Nederland.}
\email{r.kramer@uva.nl}

\author{Sergey Shadrin}
\address{Korteweg-de Vries Instituut voor Wiskunde, 
Universiteit van Amsterdam, Postbus 94248,
1090GE Amsterdam, Nederland.}
\email{s.shadrin@uva.nl}

\begin{abstract}
We use refined spectral sequence arguments to calculate known and previously unknown bi-Hamiltonian cohomology groups, which govern the deformation theory of semi-simple bi-Hamiltonian pencils of hydrodynamic type with one independent  and \( N\) dependent variables. In particular, we rederive the result of Dubrovin-Liu-Zhang that these deformations are pa\-ra\-me\-trized by the so-called central invariants, which are \( N\) smooth functions of one variable.
\end{abstract}

\maketitle

\tableofcontents

\raggedbottom


\section{Introduction}

In this paper, we consider the classification of a certain kind of dispersive evolutionary partial differential equations. More precisely, consider a convex domain \( U \subset \R^N \), and formal functions \( u: S^1 \to U \). Denoting the coordinate on \( S^1 \) by \( x\), the partial differential equations look like
\begin{equation}\label{PDE}
\frac{\d u^i}{\d t} = A_j^i(u) u_x^j + \Big( B_j^i(u)u_{xx} + C^i_{jk}(u) u^j_x u^k_x\Big) \varepsilon + \mc{O}(\varepsilon^2 )\,,
\end{equation}
as a homogeneous equation for the degree defined by \( \deg \d_t = \deg \d_x = - \deg \varepsilon = 1\).\par
We require moreover that this equation is bi-Hamiltonian and its dispersionless limit can be written as a hamiltonian equation of hydrodynamic type in two compatible ways:
\begin{equation}
\frac{\d u}{\d t} =\{ u(x),H_1\}_1 = \{ u(x), H_0\}_2\,,
\end{equation}
subject to a number of conditions that will be specified in \cref{preliminary}.\par
The archetypical example of such a structure is the Korteweg-de Vries equation, given by \( \frac{\d u}{\d t} = u u_x + \frac{\varepsilon^2}{12} u_{xxx} \), for which we have\cite{mag78}
\begin{align}
\{ u(x),u(y)\}_1 &= \delta' (x-y) & H_1 &= \int dx \Big( \frac{u^2}{2} + \frac{\varepsilon^2}{12}u_{xx} \Big) \\
\{ u(x),u(y)\}_2 &= u(x)\delta' (x-y)+\frac{1}{2} u'(x) \delta (x-y) + \frac{\varepsilon^2}{8}\delta'''(x-y) & H_0 &= \frac{2}{3} \int dx \,u\,.
\end{align}
An important reason for studying such structures is the possibility to extend it to an infinite-dimensional hierarchy of partial differential equations via the recursion operator \( \{\mathord{\cdot},\mathord{\cdot}\}_1^{-1} \{ \mathord{\cdot},\mathord{\cdot}\}_2 \):
\begin{equation}
\frac{\d u}{\d t_j} = \{ u(x), H_j \}_1 = \{ u(x), H_{j-1}\}_2\,.
\end{equation}

On the space of these structures, there is an action of the Miura group, given by diffeomorphisms of \( U\) in the dispersionless limit, with as dispersive terms differential polynomials in \( u \). Hence, it is natural to try to classify equivalence classes of Poisson pencils with respect to this group action. In 2004, in a series of papers, Dubrovin, Liu, and Zhang first considered this classification problem~\cite{lz05,dlz06}; see also~\cite{l02, z02}. In particular, they proved that the Miura equivalence class of deformations of a given semi-simple\footnote{Recently, the non-semisimple case has been considered in~\cite{dvls16}.} pencil of local Poisson brackets of hydrodynamic type is specified by a choice of $N$ functions of one variable. They called these functions \emph{central invariants}, and conjectured that for any choice of central invariants the corresponding Miura equivalence class is non-empty. This conjecture was proved in~\cite{cps15}. 

As any deformation theory of this type, its space of infinitesimal deformations as well as the space of obstructions for the extensions of infinitesimal deformations are controlled by some cohomology groups. In this case these are the so-called bi-Hamiltonian cohomology in cohomological degrees $2$ and $3$, and one should also consider there the degree with respect to the total $\d_x$-derivative, where $x$ is the spatial variable. In these terms, central invariants span the second bi-Ha\-mil\-to\-ni\-an cohomology group in $\d_x$-degree 3, and the second bi-Hamiltonian cohomology groups in $\d_x$-degrees $2$ and $\geq 4$ are equal to zero.

The computation of bi-Hamiltonian cohomology is a delicate issue. It is defined on the space of local stationary polyvector fields on the loop space of an $N$-dimensional domain $U$. A useful tool for this undertaking is the so-called $\theta$-formalism~\cite{g02}. The main technical difficulty is that we cannot immediately work with the space of densities, since there is a necessary factorization by the kernel of the integral along the loop. For the central invariants it is done in~\cite{lz05} essentially by hand for quasi-trivial pencils, i.e. pencils that are equivalent to their leading order by more general transformations, called quasi-Miura transformations. In \cite{dlz06}, it was proved that any semi-simple pencil of hydrodynamic type is quasi-trivial, completing the proof. 

In~\cite{lz13}, Liu and Zhang came up with an important new idea: they invented a way to lift the computation of the bi-Hamiltonian cohomology from the space of local polyvector fields to the space of their densities. The latter can also be considered as the functions on the infinite jet space of the loop space of the shifted tangent bundle $T_U[-1]$, independent of the loop variable $x$. Their approach was used intensively in a number of papers: it has been applied to show that the deformation of the dispersionless KdV brackets is unobstructed~\cite{lz13} and to compute the higher cohomology in this case as well~\cite{cps14}. More generally, this approach allowed a complete computation of the bi-Hamiltonian cohomology in the scalar ($N=1$) case~\cite{cps14-2}. Finally, it was used to show that the deformation theory for any semi-simple Poisson pencil is unobstructed~\cite{cps15}.

At the moment, it is not completely clear yet how widely this approach can be applied to the computation of the bi-Hamiltonian cohomology. In the case of $N>1$ the full bi-Hamiltonian cohomology is not known, and moreover, as the computation in the case $N=1$ shows, the full answer should depend on the formulas for the original hydrodynamic Poisson brackets. So far the computational techniques worked well only for the groups of relatively high cohomological grading and/or grading with respect to the total $\d_x$-derivative degree. In particular, the most fundamental result of this whole theory, the fact that the infinitesimal deformations are controlled by the central invariants, was out of reach of this technique until now. 

In this paper, we extend the computational techniques of~\cite{cps15} further and give a new proof of the theorem of Dubrovin-Liu-Zhang that the space of the Miura classes of the infinitesimal deformations of a semi-simple Poisson pencil is isomorphic to the space of $N$ functions of one variable. An advantage of our approach is that we use only the general shape of the differential induced on the jet space of $T_U[-1]$, and, for instance, the Ferapontov equations for compatible Poisson brackets of hydrodynamic type~\cite{Ferapontov01} enter the computation only through the fact that the differential squares to zero. Furthermore, our proof does not rely on the quasi-triviality theorem. A disadvantage is that in the cohomological approach of Liu-Zhang it is not possible to reproduce the explicit formula for the central invariants of a given deformation as in~\cite[Equation 1.49]{dlz06}.

\subsection{Organization of the paper}
The outline of the article is as follows. In \cref{preliminary} we recall some standard notations and formulate our main results, based on the computation of some of the cohomology of a certain complex $(\hcA[\lambda], D_\lambda)$ in the rest of the paper.
In \cref{vanishing} we give a streamlined version of the proof~\cite{cps15} of the vanishing theorem for the cohomology of $(\hcA[\lambda], D_\lambda)$.
In the next sections we proceed to compute other parts of this cohomology that will lead us in particular to the identification of the parameters of the infinitesimal deformations.
In \cref{subcomplexdCi} we compute the full cohomology of the complex $(\hat{d}_i (\hcC_i) , \cD_i )$, a subcomplex in one of the spectral sequences. In \cref{subcomplexpd} we compute the cohomology of another subcomplex, $(\hcC[\lambda],\Delta_{0,1})$, for degrees $p=d$.
In \cref{subcomplexpdpu} we prove a vanishing result in degrees $(p,d)=(3,2)$, which is essential to complete the reconstruction of the second bi-Hamiltonian cohomology group. 
In \cref{cohomologyA} we collect the results of the previous sections and, using standard spectral sequences arguments, we prove our main theorems.%

\subsection{Acknowledgments} We thank Hessel Posthuma for useful discussions and the anonymous reviewers for their comments and helpful suggestions. The authors were supported by the Netherlands Organization of Scientific Research.

\section{Recollections and main results} 
\label{preliminary}

\subsection{Poisson pencils}

Let $N$ be the number of dependent variables.
We consider a domain $U$ in $\mathbb{R}^N$ outside the diagonals. Let $u^1,\dots,u^N$ be the coordinate functions of $\mathbb{R}^N$ restricted to $U$. We denote the corresponding basis of sections of $T_U[-1]$ by $\theta^0_1,\dots,\theta^0_N$. We denote by $\mc{A}$ the space of functions on the jet space of the loop space of $U$ that do not depend on the loop variables $x$, that is, 
\begin{equation}
\mc{A} \coloneq C^\infty(U) \Big\llbracket \left\{u^{i,d}\right\}_{\substack{i=1,\dots, N \\ d=1,2,\dots}}  \Big\rrbracket\,,
\end{equation}
and we call its elements differential polynomials. \par
Similarly, we denote by $\hcA$ the space of functions on the jet space of the loop space of $T_U[-1]$ that do not depend on the loop variables $x$, 
\begin{equation}
\hcA \coloneq C^\infty(U) \Big\llbracket \left\{u^{i,d}\right\}_{\substack{i=1,\dots, N \\ d=1,2,\dots}} , \left\{\theta_i^{d}\right\}_{\substack{i=1,\dots, N \\ d=0,1,2,\dots}} \Big\rrbracket\,.
\end{equation}
Sometimes it is convenient to denote the coordinate functions $u^i$ by $u^{i,0}$, for $i=1,\dots,N$. 

The standard derivation, i.e., the total derivative with respect to the variable $x$, is given by 
\begin{equation}
\d_x \coloneq \sum_{d=0}^\infty \left( u^{i,d+1}\frac{\d}{\d u^{i,d}} + \theta^{d+1}_i \frac{\d}{\d \theta^d_i} \right)\,,
\end{equation}
where we assume summation over the repeated basis-related indices (here \( i\)).
\begin{definition}
The space of \emph{local functionals} on \( U\) is defined to be \( \hat{\mc{F}} \coloneq \hat{\mc{A}}/\d_x \hat{\mc{A}} \). The natural quotient map is denoted \( \int dx \colon \hat{\mc{A}} \to \hat{\mc{F}} \).
\end{definition}
Note that both spaces $\hcA$ and $\hcF$ have two gradations: the standard gradation that we also call the $\d_x$-degree in the introduction, given by $\deg u^{i,d}=\deg \theta_i^d=d$, $i=1,\dots,N$, $d\geq0$, and the super gradation that we also call the cohomological or the $\th$-degree, given by $\deg_\th  u^{i,d}= 0$, $\deg_\th \th^d_i = 1$, $i=1,\dots,N$, $d\geq 0$. The first degree is also defined on \( \mc{A} \). We denote by $\hcA^p_d$ (respectively, $\hcF^p_d$) the subspace of $\hcA$ (respectively, $\hcF$) of standard degree $d$ and cohomological degree $p$.
\begin{definition}
A \emph{(dispersive) Poisson pencil} is a pair of Poisson brackets \( \big\{ \{\plh, \plh\}_a \big\}_{a=1,2} \) on \( \mc{F}\llbracket \varepsilon \rrbracket \), homogeneous of standard degree one, where \( \deg \varepsilon = -1\), such that \( \{ \plh,\plh \}_2 + \l \{\plh,\plh \}_1\) is a Poisson bracket for any \( \l \in \R \).\par
A \emph{dispersionless Poisson pencil} is a dispersive Poisson pencil which does not depend on \( \varepsilon \). Any dispersive Poisson pencil has a \emph{dispersionless limit}: this is the constant term in \( \varepsilon \).\par
We will furthermore implicitly require all our Poisson pencils to have a \emph{hydrodynamic dispersionless limit} on \( \mc{F} \),
\begin{equation}
\{ u^i(x),u^j(y)\}_a = \big( g^{ij}_a(u)\d_x + \Gamma^{ij}_{k,a}(u)u^k_x\big) \delta (x-y) + \mc{O}(\varepsilon )\,.
\end{equation}
\end{definition}
\begin{remark}
For any Poisson bracket of hydrodynamic type, \( g^{ij}_a \) is a flat pseudo-Riemannian metric on \( U\) with Christoffel symbols \( \Gamma^{ij}_{k,a} \), as proved by Dubrovin and Novikov in \cite{dn83}.
\end{remark}
\begin{definition}
A Poisson pencil of hydrodynamic type is \emph{semi-simple} if the eigenvalues of \( g^{ij}_2 - \l g^{ij}_1 \) are all distinct and non-constant on \( U\).
\end{definition}
From now on, we will assume the dispersionless limit of our Poisson pencils are semi-simple, and use the roots of \( \det( g_2^{ij} - \l g_1^{ij}) \) as canonical coordinates \( u^i\) on \( U\). This reduces the metrics to
\begin{align}
g_1^{ij}(u) &= f^i(u)\delta^{ij} & g_2^{ij}(u) = u^if^i(u)\delta^{ij}\,,
\end{align}
for $N$ non-vanishing functions $f^1,\dots,f^N$, subject to the following equations derived by Ferapontov~\cite{Ferapontov01}. Let $H_i:=(f^i)^{-1/2}$, $i=1,\dots,N$, be Lam\'e coefficients and $\gamma_{ij}:= (H_i)^{-1} \d_i H_j$, $i\not=j$, be rotation coefficients for the metric determined by $f^1,\dots,f^N$. Here we denote by $\d_i$ the derivative $\d/\d u^i$.  Then we have:
\begin{align}
& \d_k \gamma_{ij} = \gamma_{ik}\gamma_{kj}, & i\not= j\not= k\not= i; \label{eq:dgammaijk}\\
& \d_i \gamma_{ij} + \d_j \gamma_{ji} +\sum\nolimits_{k\not=i,j} \gamma_{ki} \gamma_{kj} =0, & i\not=j; \label{dgammaij}\\
& u^i\d_i \gamma_{ij} + u^j\d_j \gamma_{ji} +\sum\nolimits_{k\not=i,j} u^k\gamma_{ki} \gamma_{kj}
+\frac 12 (\gamma_{ij}+\gamma_{ji}) = 0, & i\not=j . \label{udgamma}
\end{align}
Note that there is no implicit summation in these equations, as these only occur in the case of contractions of generators of \( \hcA \) and derivatives with respect to them, and are a shorthand for matrix-like multiplications. In the rest of the paper, we will often include an explicit summation sign if there is a chance of confusion. If in doubt about an implicit summation, it will suffice to check the other side of the equation for occurence of the same summation index.\par
The space of Poisson pencils has a naturally-defined automorphism group:

\begin{definition}
The \emph{Miura group} is the group of transformations of the form
\begin{equation}
u^i \mapsto v^i(u) + \sum_{k \geq 1}\Phi^i_k \varepsilon^k\,,
\end{equation}
where \( v\) is a diffeomorphism of \( U\) and the \( \Phi_k^i \) are differential polynomials of degree \( k\). Hence the total degree of any Miura transformation is zero.
\end{definition}
Given this action, it is a natural question to try to classify Poisson pencils up to equivalence. Choosing canonical coordinates as above fixes the leading term of the Miura transformation (the transformation of \emph{first type}), so the remaining freedom is given by transformations with \( v = \Id_U \) (transformations of the \emph{second type}). The first main result to answer this question is the following theorem by Dubrovin, Liu, and Zhang:
\begin{theorem}[\cite{lz05,dlz06}]\label{centinv}
Given a dispersionless Poisson pencil \( \{ \plh ,\plh \}_a^0\), deformations of the form
\begin{equation}
\{ u^i(x),u^j(y)\}_a = \{ u^i(x),u^j(y)\}_a^0 + \sum_{k\geq 1} \varepsilon^k \sum_{l=0}^{k+1} A_{k,l;a}^{ij}\delta^{(l)}(x-y)\,,
\end{equation}
where \( A^{ij}_{k,l;a} \) are differential polynomials of degree \( k+1-l \), are equivalent if and only if the following associated functions, called \emph{central invariants}, are equal:
\begin{equation}
c_i(u) \coloneq \frac{1}{3(f^i(u))^2}\Bigg( A_{2,3;2}^{ii} - u^i A_{2,3;1}^{ii} + \sum_{k \neq i} \frac{\big( A_{1,2;2}^{ki} - u^i A_{1,2;1}^{ki} \big)^2}{f^k(u)(u^k-u^i)}\Bigg)\,.
\end{equation}
Furthermore, \( c_i \) only depends on \( u^i \).
\end{theorem}
They also conjectured that any set of such functions has an associated deformation class. This conjecture was settled recently:
\begin{theorem}[\cite{cps15}]\label{defexist}
Given a dispersionless Poisson pencil \( \{ \plh, \plh \}_a^0\) and a  set \( \big\{ c_i(u) \in C^\infty (U)\big\}_{i=1}^N \), such that each \( c_i \) depends only on \( u^i \), there exists a deformation of the pencil as in the previous theorem which has the \( c_i \) as central invariants. 
\end{theorem}
The first theorem was proved using quasi-triviality of Poisson pencils, involving Miura transformations with rational differential functions, i.e. the dependence on the \( u^{i,d} \) is allowed to be rational. The second theorem used more general methods from homological algebra, using formalism and techniques developed by Liu and Zhang \cite{lz13}. The main result of the current paper is an extension of the results of \cite{cps15}, which in particular also implies the abstract form of \cref{centinv}, showing that deformations of a dispersionless Poisson pencil are classified by \( N\) smooth functions, each dependent on one \( u^i \). Hence, this paper gives a unified proof of both theorems, yielding a complete classification of deformations of Poisson pencils of hydrodynamic type in several dependent and one independent variable, with the caveat that the explicit form of the central invariants cannot be recovered by this method.\par

\subsection{Cohomological formulation}
In essence, the theorems in the previous subsection are cohomological statements: \cref{centinv} states that \emph{infinitesimal deformations}, i.e., deformations up to \( \mc{O}(\varepsilon^3 )\), are equivalent if and only if their central invariants are, and can be extended to at most one deformation to all orders, while \cref{defexist} states that this deformation to all orders exists. To develop the right cohomological notions, we have to introduce some more notation.
\begin{definition}
On \( \hcA \), the \emph{variational derivatives} with respect to the coordinates on \( T_U[-1] \) are defined via the Euler-Lagrange formula as
\begin{align}
\frac{\delta}{\delta u^i} &= \sum_{s \geq 0} (-\d )^s \frac{\d}{\d u^{i,s}}\,, & \frac{\delta}{\delta \th_i} &= \sum_{s \geq 0} (-\d )^s \frac{\d}{\d \th_i^s}\,.
\end{align}
These are zero on total \( \d_x \)-derivatives, so they factor through maps \( \hcF \to \hcA \), which we denote by the same symbols.\par
The \emph{Schouten-Nijenhuis bracket} is defined by
\begin{equation}
[\plh ,\plh ] : \hcF^p \times \hcF^q \to \hcF^{p+q-1} : \bigg( \int \! A \, dx, \int \! B\, dx \bigg) \mapsto \int  \bigg( \frac{\delta A}{\delta \th_i} \frac{\delta B}{\delta u^i} + (-1)^p \frac{\delta A}{\delta u^i} \frac{\delta B}{\delta \th_i} \bigg) dx \,.
\end{equation}
\end{definition}
In a completely analogous way to the finite-dimensional case, a Poisson bracket \( \{ \plh, \plh \} \) corresponds to a bivector \( P \in \hcF^2 \) such that \( [P,P] = 0\), and therefore induces a differential \( d_P = [P,\plh ]\) on \( \hcF \). This can be lifted straightforwardly to a differential \( D_P \) on \( \hcA \).\par
For a pencil \( \{ \plh, \plh \}_a \), we get \( P_a \in \hcF \) such that \( d_{P_1} d_{P_2} + d_{P_2} d_{P_1} = 0\), so \( d_\l = d_{P_2} - \l d_{P_1} \) is a differential on \( \hcF [\l ]\), and similarly, \( D_\l \) is one on \( \hcA [\l ] \). Explicitly, for a pencil given by the functions $f^1,\dots,f^N$, $D_\lambda$ is defined as $D_\lambda:=D(u^1f^1,\dots,u^Nf^N)-\lambda D(f^1,\dots,f^N)$, where
\begin{align}
\label{diff-biham}
D(g^1,\dots,g^N) = &
\sum_{s\geq0}  \partial^s\left(g^i\theta_i^1\right)\frac{\d}{\d u^{i,s}}
\\ \notag 
& +\frac{1}{2}\sum_{s\geq0}  \partial^s\bigg(
\d_jg^i u^{j,1} \theta_i^0 
+ g^i\frac{\d_i g^j}{g^j} u^{j,1}\theta_j^0 
- g^j\frac{\d_j g^i}{g^i} u^{i,1}\theta_j^0
\bigg) \frac{\d}{\d u^{i,s}}
\\
& + \frac{1}{2} \sum_{s\geq0}  \partial^s \bigg( 
\d_ig^j  \theta_j^0 \theta_j^1 
+ g^j\frac{\d_j g^i}{g^i} \theta_i^0 \theta_j^1 
- g^j\frac{\d_j g^i}{g^i} \theta_j^0 \theta_i^1
\bigg) \frac{\d}{\d \theta_i^s}
\\
&+\frac12 \sum_{s\geq0} \d^s \bigg( 
\d_i \Big( g^k \frac{\d_k g^j}{g^j} \Big) u^{j,1} \theta^0_k \theta^0_j  
-\d_j \Big( g^k \frac{\d_k g^i}{g^i} \Big) u^{j,1} \theta^0_k \theta^0_i \bigg) \frac{\partial }{\partial \theta_i^s} .
\end{align}
By a result of \cite{dz01,g02,dms05}, \( H^2 (\hcF, d_P ) = 0\) for any hydrodynamic Poisson bivector \( P\). This makes it possible to construct, order by order, a Miura tranformation that turns the first Poisson bracket in a deformed Poisson pencil into its dispersionless part. Hence, to deform the second bracket, we should consider the following:
\begin{definition}[\cite{dz01}]
The \emph{bi-Hamiltonian cohomology} of a Poisson pencil \( P_1\), \( P_2 \) is
\begin{equation}
BH(U,P_1,P_2) = \frac{\ker d_{P_1} \cap \ker d_{P_2}}{\im d_{P_1}d_{P_2}}\,.
\end{equation}
As in similar cases, we denote by $BH^p_d$ the subspace of $BH$ of $\d_x$-degree $d$ and cohomological degree $p$.
\end{definition}
An interpretation of the first few of these groups has also been given in \cite{dz01}:
\begin{itemize}
\item The common Casimirs of the Poisson pencil are given by \( BH^0\);
\item The bi-Hamiltonian vector fields are given by \( BH^1\);
\item The equivalence classes of infinitesimal deformations of the pencil are given by \( BH^2_{\geq 2} \);
\item The obstruction to extending infinitesimal deformations to deformations of a higher order are given by \( BH^3_{\geq 5}\).
\end{itemize} 

We can restate \cref{centinv,defexist} together using bi-Hamiltonian cohomology. We denote by $C^\infty (u^i)$ the space of smooth functions on $U$ that only depend on the single variable $u^i$. 

\begin{theorem} \label{thm:mainstatement} We have $BH^2_d$ is equal to zero for $d=2$ and $d\geq 4$. In the case $d=3$, $BH^2_3$ is isomorphic to $\bigoplus_{i=1}^N C^\infty (u^i)$. Moreover, \( BH_d^3 \) is zero for \( d \geq 5\).
\end{theorem}
This is the form of the theorem of which we will give a uniformized proof in this paper. We will actually prove the more general Theorem~\ref{thm:mainmain}, from which this theorem follows.

In order to calculate the bi-Hamiltonian cohomology, we use the key lemma of~\cite{lz13}, see also~\cite{b08}, which implies that for $d\geq 2$ we have that $BH^p_d\cong H^p_d(\hcF[\lambda], d_\lambda)$. Another idea of Liu and Zhang~\cite{lz13} is that in order to compute the cohomology of $(\hcF[\lambda], d_\lambda)$ one might use the long exact sequence in the cohomology induced by the short exact sequence
\begin{equation}
0 \to (\hcA[\lambda]/\R[\lambda], D_\lambda) \xrightarrow{\d_x} ( \hcA[\lambda], D_\lambda) \to (\hcF [\lambda], d_\lambda) \to 0\,.
\end{equation}
In particular, we will consider the parts of the form
\begin{equation}\label{eq:MainExact}
H^p_{d-1} (\hcA[\l]) \to H^p_{d}(\hcA[\l]) \to H^p_d(\hcF[\l]) \to H^{p+1}_{d} (\hcA[\l]) \to H^{p+1}_{d+1}(\hcA[\l])
\end{equation}
for \( d \geq 2\). We omit the differentials in the notation for the cohomology since they are always $D_\lambda$ for the space $\hcA[\lambda]$ and $d_\lambda$ for the space $\hcF[\l]$.

We want to derive \cref{thm:mainstatement} from the exact sequence given by \cref{eq:MainExact}. In order to do this, let us recall that in~\cite{cps15} the following vanishing theorem for the cohomology of the complex $(\hcA[\lambda], D_\lambda)$ was proved. 
\begin{theorem} \label{vanth}
	The cohomology $H^p_d(\hcA[\lambda])$ vanishes for all bi-degrees $(p,d)$, unless $(p,d)=(d+k,d)$ with
	\begin{equation}
	k=0, \dots , N-1 , \quad  d=0, \dots ,N+2  \quad 
	\text{or} \quad k=N, \quad d=0, \dots , N.
	\end{equation}
\end{theorem}

We give a streamlined proof of this theorem in the next section.\par

The main contributions of this paper are the following results about the cohomology of $\hcA[\lambda]$.
\begin{theorem}\label{resultp=d}
For \( p =d \), the cohomology of \( \hcA [\l ] \) is given by:
\begin{equation}
H_p^p (\hcA [\l ], D_\l ) \cong \begin{cases}
\R [\l ] & p = 0,\\
\bigoplus_{i=1}^N C^\infty (u^i) \th_i^0 \th_i^1 \th_i^2 & p = 3,\\
0 & \text{else.}
\end{cases}
\end{equation}
\end{theorem}

\begin{theorem}\label{extravanishing}
The cohomology  \( H^p_d(\hcA [\l ], D_\lambda ) \) vanishes for 
\begin{equation}
\begin{cases}
p<d, \quad &d\geq 0\,; \\
p>d+N, \quad &d\geq 0\,; \\
d<p\leq d+N, \quad  &d>\max(3,N)\,; \\
p=3, &d=2\,.
\end{cases}
\end{equation}
\end{theorem}


Assuming these theorems, we can formulate our main result on the bi-Hamiltonian cohomology, from which Theorem~\ref{thm:mainstatement} follows:

\begin{theorem} \label{thm:mainmain}
The bi-Hamiltonian cohomology \( BH_d^p \) vanishes for
\begin{align}
\begin{cases}
p < d & d \geq 2\,; \\
p > d+N & d \geq 2\,; \\
d \leq p \leq d+N & d > \max (3,N)\,; \\
p = 2 & d = 2\,,
\end{cases}
\end{align}
unless \( (p,d) = (2,3)\), in which case it is isomorphic to \( \bigoplus_{i=1}^N C^\infty (u^i) \), the space of central invariants.
\end{theorem}

The regions of this theorem are visualized in \cref{vanishingregions}.
\begin{figure}[ht]
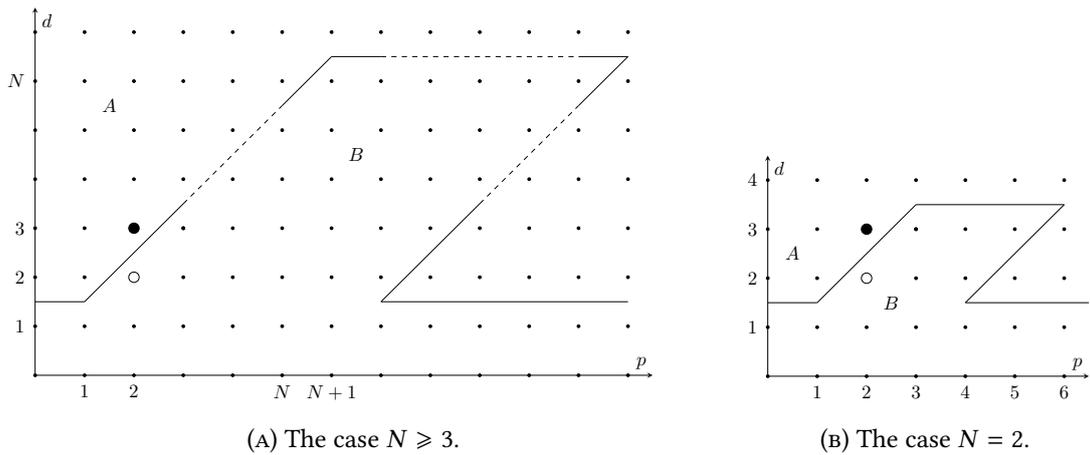

\begin{subfigure}[b]{0.66\textwidth}
\begin{lpic}{VanishingregionNgreaterthan2(0.65,0.65)}
\end{lpic}
\caption{The case \( N \geq 3\).}
\end{subfigure}
~
\begin{subfigure}[b]{0.33\textwidth}
\begin{lpic}{VanishingregionN=2(0.65,0.65)}
\end{lpic}
\caption{The case \( N = 2\).}
\end{subfigure}
\caption{All bi-Hamiltonian cohomology groups are zero in region \( A\), except for the black dot, which is given by the central invariants. All groups are unknown in region \( B\), except for the white dot, which vanishes.}\label{vanishingregions}
\end{figure}

\begin{proof}
Using the isomorphism between \( BH_d^p \) and \( H_d^P (\hcF [\l ] )\) in the required range, all the vanishing statement follow from the exact sequence~\eqref{eq:MainExact} as both the second and the fourth terms are zero. For $(p,d)=(2,3)$, the second term is zero, which implies that $H^2_3(\hcF[\lambda])\cong H^3_3(\hcA[\lambda])$, and $H^3_3(\hcA)\cong \bigoplus_{i=1}^N C^\infty (u^i)$ by \cref{resultp=d}.
\end{proof}


\begin{remark}
Observe that the cohomology of \( \hcA [\l ] \) is still unknown on the subcomplexes $p=d+1, \dotsc, d + N$ for $d< N$, unless  $(p,d)=(3,2)$ or unless \( N =1 \). The last case has been determined completely in \cite[Proposition 4]{cps14-2}. The key to determining the cohomology completely would likely lie in an extension of the proof of \cref{prop32}, where one would have to study more carefully the transformation \( \th_i^0 \mapsto \bar{\th}_i^0 \).  This transformation is trivial in the case \( N =1 \), so the subtlety does not occur there.
\end{remark}

We conclude this section with one more piece of notation that we use throughout the rest of the paper: for a multi-index \( I = \{ i_1, \dotsc, i_s\} \), we write \( f^I = \prod_{i \in I} f^i\), \( \theta_I^t = \theta_{i_1}^t \dotsb \theta_{i_s}^t \),   etc.

\section{The first vanishing theorem}
\label{vanishing}

In this section we give a proof of \cref{vanth}, based on the proof of \cite{cps15}. This section does not contain any new results, but has the main purpose of recalling some objects that will be used later.\par
The presentation of the proof given here is improved over \cite{cps15}, mainly by focusing less on the intricacies of spectral sequences and more on the structure and decomposition of the spaces and differentials involved. This exposition is somewhat less detailed as a result and the reader is expected to be familiar with spectral sequence techniques for graded complexes; more details can be found in \cite{cps15}.

\subsection{}
\label{degu}
Let $\deg_u$ be the degree on $\hcA$ defined by assigning
\begin{equation}
\deg_u u^{i,s} = 1, \quad  s >0
\end{equation}
and zero on the other generators. The operator $D_\lambda$ splits in the sum of its homogeneous components
\begin{equation}
D_\lambda = \Delta_{-1} + \Delta_0 + \dots \,,
\end{equation}
where $\deg_u \Delta_k= k$.

To the degree $\deg_u + \deg_\theta$ we associate a decreasing filtration of $\hcA[\lambda]$. Let us denote by $\tensor*[^1]{E}{}$ the associated spectral sequence. The zero page $\tensor*[^1]{E}{_0}$ is simply given by $\hcA[\lambda]$ with differential $\Delta_{-1}$:
\begin{equation}
(\tensor*[^1]{E}{_0} , \tensor*[^1]{d}{_0} ) = ( \hcA[\lambda], \Delta_{-1} ).
\end{equation}
To find the first page $\tensor*[^1]{E}{_1}$, we have to compute the cohomology of this complex.

\subsection{}
Let us compute the cohomology of the complex 
$(\hcA[\lambda], \Delta_{-1})$. 
The differential can be written as
\begin{equation}
\Delta_{-1} = \sum_i (-\lambda+u^i) f^i \hat{d}_i  
\end{equation}
where $\hat{d}_i$ is the de Rham-like differential
\begin{equation}
\hat{d}_i = \sum	_{s \geq 1} \theta_i^{s+1} \frac{\partial }{\partial u^{i,s}} .
\end{equation}

It is convenient to split $\hcA$ in a direct sum
\begin{equation}
\hcA = \hcC \oplus \left( \bigoplus_{i=1}^N  \hcC_i^\textnormal{nt} \right) \oplus \hcM.
\end{equation}
Here 
\begin{equation}
\hcC=C^\infty(U)[\theta^0_1,\dots,\theta^0_N,\theta^1_1,\dots,\theta^1_N],
\end{equation}
and
\begin{equation}
\hcC_i = \hcC \llbracket  \{ u^{i,s}, \theta_i^{s+1} \mid s\geq1 \} \rrbracket,
\end{equation}
while $\hcC_i^\textnormal{nt}$  denotes the subspace of $\hcC_i$ spanned by nontrivial monomials, i.e., all monomials that contain at least one of the variables $u^{i,s}$, $\theta_i^{s+1}$ for $s\geq 1$. 
By $\hcM$ we denote the subspace of $\hcA$ spanned by monomials which contain at least one of the mixed quadratic expressions
\begin{equation}
u^{i,s} u^{j,t}, \quad u^{i,s} \theta_j^{t+1}, \quad \theta_i^{s+1} \theta_j^{t+1}
\end{equation}
for some $s,t \geq1$ and $i\not= j$. 

\begin{lemma}
The differential $\Delta_{-1}$ leaves invariant each direct summand in 
\begin{equation} \label{split1}
\hcA[\lambda] = \hcC[\lambda] \oplus \left( \bigoplus_{i=1}^N  \hcC_i^\textnormal{nt}[\lambda] \right) \oplus \hcM[\lambda],
\end{equation}
and in particular maps $\hcC[\lambda]$ to zero.
\end{lemma}
\begin{proof}
It is easy to check that 
\begin{align}
&\hat{d}_i (\hcC) = 0 , 
&\hat{d}_i (\hcM) \subseteq \hcM, \\
&\hat{d}_i (\hcC_i^\textnormal{nt}) \subseteq \hcC_i^\textnormal{nt}, 
&\hat{d}_i (\hcC_j^\textnormal{nt} ) = 0 \quad i\not= j, 
\end{align}
from which the lemma follows immediately.
\end{proof}

The cohomology of $\hcA[\lambda]$ is therefore the direct sum of the cohomologies of the summands in the direct sum~\eqref{split1}, and in particular
\begin{equation}
H(\hcC[\lambda], \Delta_{-1}) = \hcC[\lambda] .
\end{equation}

Let us first observe that the cohomology of the de Rham complex $(\hcC_i, \hat{d}_i)$ is trivial in positive degree.
\begin{lemma}
\begin{equation}
H(\hcC_i, \hat{d}_i) = \hcC. 
\end{equation}
\end{lemma}
\begin{proof}
The proof is completely analogous to the standard proof of the Poincar\'e lemma.
\end{proof}

In particular we have that 
\begin{equation} \label{poincare1}
H(\hcC_i^\textnormal{nt}, \hat{d}_i ) = 0 ,
\end{equation}
therefore the kernel of $\hat{d}_i$ in $\hcC_i^\textnormal{nt}$ coincides with $\hat{d}_i(\hcC_i)$.

\begin{lemma}
\begin{equation}
H(\hcC_i^\textnormal{nt}[\lambda], \Delta_{-1} ) = \frac{\hat{d}_i (\hcC_i)[\lambda]}{(-\lambda+u^i)\hat{d}_i (\hcC_i)[\lambda]}.
\end{equation}
\end{lemma}
\begin{proof}
On $\hcC_i^\textnormal{nt}[\lambda]$ the differential $\Delta_{-1}$ is equal to $(-\lambda+u^i) f^i \hat{d}_i$. Its kernel coincides with the kernel of $\hat{d}_i$ on $\hcC_i^\textnormal{nt}[\lambda]$, which is $d_i(\hcC_i)[\lambda]$. Its image is $(-\lambda+u^i)\hat{d}_i (\hcC_i)[\lambda]$.
\end{proof}

Finally we prove that the complex $(\hcM[\lambda], \Delta_{-1})$ is acyclic.
\begin{lemma}
\begin{equation}
H(\hcM[\lambda], \Delta_{-1}) =0.
\end{equation}
\end{lemma}
\begin{proof} This lemma can be proved by induction on $N$. Denote, for convenience, the corresponding space and the differential by $\hcM[\lambda]_{(N)}$ and $\Delta_{-1,(N)}$. We also use in the proof the notation $\hcA_{(N)}$ and $\hcC_{(N)}$. 
	
The differential $\Delta_{-1}$ is naturally the sum of two commuting differentials, 
\begin{equation}
	\Delta_{-1,(N)}=\Delta_{-1,(N-1)}+(-\lambda+u^N) f^N \hat d_N.
\end{equation} 
	The cohomology of $(-\lambda+u^N) f^N \hat{d}_N$ on  $\hcM[\lambda]_{(N)}$ is equal to the direct sum of two subcomplexes, $\hcC_{(N)}\otimes_{\hcC_{(N-1)}}\hcM[\lambda]_{(N-1)}$ and 
	$$
	\frac{\hat d_N(\hcC^\textnormal{nt}_N)\otimes_{\hcC_{(N)}} \left(\left( \bigoplus_{i=1}^{N-1}  \hcC_i^\textnormal{nt}[\lambda] \right) \oplus\hcC_{(N)}\otimes_{\hcC_{(N-1)}}\hcM[\lambda]_{(N-1)}\right)}{(-\lambda+u^N)}. 
	$$
	On the first component the induced differential is equal to $\Delta_{-1,(N-1)}$, so we can use the induction assumption. On the second component the induced differential is equal to $$\left.\left(\Delta_{-1,(N-1)}\right)\right|_{\lambda=u^N},$$ 
	so, up to rescaling by non-vanishing functions, it is a de Rham-like differential acting only on the second factor of the tensor product. This second factor can be identified with $\hcC_{(N)}\otimes_{\hcC_{(N-1)}} \hcA_{(N-1)}/\hcC_{(N-1)}$, so the possible non-trivial cohomology is quotiented out (cf. the standard proof of the Poincar\'e lemma).
\end{proof}

This completes the computation of the cohomology of the complex $(\hcA[\lambda],\Delta_{-1})$:
\begin{proposition}
\begin{equation} \label{cohA}
H(\hcA[\lambda], \Delta_{-1}) = \hcC[\lambda] \oplus \left(
\bigoplus_{i=1}^N \frac{\hat{d}_i (\hcC_i)[\lambda]}{(-\lambda+u^i)\hat{d}_i (\hcC_i)[\lambda]} 
\right)
\end{equation}
\end{proposition}

%
%
%
%

\subsection{}\label{sec:1E1}

The first page $\tensor*[^1]{E}{_1}$ of the first spectral sequence is given by the cohomology of the complex $H(\hcA[\lambda], \Delta_{-1})$ with the differential induced by the operator $\Delta_0$:
\begin{equation}
(\tensor*[^1]{E}{_1} , \tensor*[^1]{d}{_1} ) = (H(\hcA[\lambda], \Delta_{-1}) , \Delta_0 ).
\end{equation}
We recall the formula for the operator $\Delta_0$ in the appendix.
To get the second page $\tensor*[^1]{E}{_2}$ of the first spectral sequence we have to compute the cohomology of this complex.

Let $\deg_{\theta^1}$ be the degree on $\hcA$ defined by setting
\begin{equation}
\deg_{\theta^1} \theta^1_i =1 \quad i=1, \dots ,N
\end{equation}
and zero on the other generators. The operator $\Delta_0$ splits in its homogeneous components
\begin{equation}
\Delta_0 = \Delta_{0,1} + \Delta_{0,0}+\Delta_{0,-1}
\end{equation}
where $\deg_{\theta^1} \Delta_{0,k} =k$.

To the degree $\deg_{\theta^1} - \deg_{\theta}$ we associate a decreasing filtration of $H(\hcA[\lambda], \Delta_{-1})$, and denote by $\tensor*[^2]{E}{}$ the associated spectral sequence. 
The zero page $\tensor*[^2]{E}{_0}$ is given by $H(\hcA[\lambda], \Delta_{-1})$ with the differential induced by $\Delta_{0,1}$:
\begin{equation}
(\tensor*[^2]{E}{_0} , \tensor*[^2]{d}{_0} ) = (H(\hcA[\lambda], \Delta_{-1}) , \Delta_{0,1} ).
\end{equation}
The first page $\tensor*[^2]{E}{_1}$ is given by the cohomology of this complex.

\subsection{}
\label{D01}

To obtain a simple expression for the action of $\Delta_{0,1}$ on the cohomology~\eqref{cohA}, it is convenient to perform a change of basis in $\hcA$.
Let $\Psi$ be the invertible operator that rescales the generators of $\hcA$ as follows
\begin{equation}
u^{i,s} \mapsto (f^i)^{\frac{s}2} u^{i,s} , \quad \theta_i^s \mapsto (f^i)^{\frac{s+1}2} \theta_i^s .
\end{equation}
The operator $\Delta_{0,1}$ has a  simpler form when conjugated with $\Psi$, and since $\Psi$ leaves invariant all the subspaces that we consider, such conjugation does not affect the computation of the cohomology.

\begin{lemma} \label{lemmaD01}
The operator $\Delta_{0,1}$ acts on the cohomology~\eqref{cohA} as $\Psi \tilde{\Delta}_{0,1} \Psi^{-1}$, where
\begin{equation}
\tilde\Delta_{0,1} =\sum_i (-\lambda+u^i) \theta_i^1 \frac{\partial }{\partial u^i} 
+\sum_{i,j} (-\lambda+u^j) (\gamma_{ij} \theta_j^1 - \gamma_{ji} \theta_i^1 ) \theta^0_j \frac{\partial }{\partial \theta_i^0} 
+\sum_i \theta_i^1 \cE_i
\end{equation}
and leaves invariant each of the summands in \cref{cohA}. Here $\cE_i$ is the Euler operator that multiplies any monomial $m$ by its weight $w_i(m)$ defined by
\begin{equation}
w_i(u^{i,s}) = \frac{s}2 +1 , \quad w_i(\theta_i^{s-1}) = \frac{s}2-1 \quad s\geq1
\end{equation}
and zero on the other generators.

\end{lemma}
\begin{proof}
Recall that $\Delta_{0,1}$ is the $\deg_u =0$ and $\deg_{\theta^1} = 1$ homogeneous component of the differential $D_\lambda$. An explicit expression can be found in~\cite{cps15}. 
By a straightforward computation, we have that $\Psi^{-1} \Delta_{0,1} \Psi$ is equal to $\tilde{\Delta}_{0,1}$ plus two extra terms
\begin{align}
&-\sum_{i,j}\sum_{s\geq1} (-\lambda+u^i) \left( \frac{f^i}{f^j} \right)^{\frac{s+1}2} \left(
(s+2) \gamma_{ji} \theta_i^1 +s \gamma_{ij}\theta_j^1 
\right) u^{j,s} \frac{\partial }{\partial u^{i,s} } \\
&+\sum_{i,j}\sum_{s\geq2} (-\lambda+u^j) \left( \frac{f^i}{f^j} \right)^{\frac{s}2} \left(
(1-s) \gamma_{ij} \theta_j^1-(1+s) \gamma_{ji} \theta_i^1
\right) \theta_j^s \frac{\partial }{\partial \theta_i^s}. 
\end{align}
The following formulas are useful in the computation of the conjugated operator:
\begin{equation}
\Psi^{-1} \frac{\partial }{\partial u^{i,s}} \Psi	 = (f^i)^{\frac{s}2}  \frac{\partial }{\partial u^{i,s}} ,
\quad  
\Psi^{-1} u^{i,s} \Psi = (f^i)^{-\frac{s}2} u^{i,s},
\end{equation}
\begin{equation}
\Psi^{-1} \frac{\partial }{\partial \theta_i^s} \Psi	 = (f^i)^{\frac{s+1}2}  \frac{\partial }{\partial \theta_i^s} ,
\quad  
\Psi^{-1} \theta_i^s \Psi = (f^i)^{-\frac{s+1}2} \theta_i^s,
\end{equation}
\begin{equation}
\Psi^{-1} \frac{\partial }{\partial u^i} \Psi= \frac{\partial }{\partial u^i} + \sum_{j} \frac{\partial \log f^j}{\partial u^i} \sum_{s\geq0} \left(
\frac{s}2 u^{j,s} \frac{\partial }{\partial u^{j,s}} + \frac{s+1}2 \theta_j^s \frac{\partial }{\partial \theta_j^s}  
\right)   .
\end{equation}

By construction the operator $\Delta_{0,1}$ induces a map on the cohomology~\eqref{cohA}, and so does the conjugated operator $\Psi^{-1} \Delta_{0,1} \Psi$.

Let us make a few easy to check observations in order to simplify this operator:
\begin{enumerate}
\item
$\tilde{\Delta}_{0,1}$ maps $\hcC[\lambda]$ to itself, while the two extra terms send it to zero;

\item 
the two extra terms, when $j\not= i$, send $\hat{d}_i(\hcC_i)[\lambda]$ to $\hcM[\lambda]$ which is trivial in cohomology;

\item 
both $\tilde{\Delta}_{0,1}$ and the extra terms for $j=i$ map $\hat{d}_i(\hcC_i)[\lambda]$ to $\hcC^{\textup{nt}}_i[\lambda]$, and, because they need to act on cohomology, they actually send it to $\hat{d}_i(\hcC_i)[\lambda]$;

\item
terms in $\hat{d}_i(\hcC_i)[\lambda]$ which are proportional to $\lambda - u^i$ actually vanish in cohomology, so we can set $\lambda$ equal to $u^i$; this in particular kills the $i=j$ part of the extra terms. 
\end{enumerate}
The lemma is proved.
\end{proof}

Let us identify
\begin{equation}
\label{ide}
\frac{\hat{d}_i (\hcC_i)[\lambda]}{(-\lambda+u^i)\hat{d}_i (\hcC_i)[\lambda]} \simeq \hat{d}_i (\hcC_i)
\end{equation}
by setting $\lambda$ equal to $u^i$.  
Let $\cD_i$ be the operator induced by $\Delta_{0,1}$ on $\hat{d}_i (\hcC_i)$ by this identification. Its explicit form is given in the next corollary.

\begin{corollary}
The operator $\cD_i$ on $\hat{d}_i (\hcC_i)$ is given by $\cD_i= \Psi \tilde{\cD}_i  \Psi^{-1}$ with 
\begin{equation}
\tilde{\cD}_i =\sum_k \theta_k^1 \left[ (u^k - u^i) \left( \frac{\partial }{\partial u^k} + \sum_j \gamma_{jk} \theta_k^0 \frac{\partial }{\partial \theta_j^0}  \right)
+ \sum_j (u^i - u^j) \gamma_{jk} \theta_j^0 \frac{\partial }{\partial \theta_k^0} 
+\cE_k \right].
\end{equation}
\end{corollary}

The first page of the second spectral sequence is therefore given by the following direct sum
\begin{equation} \label{2E1}
\tensor*[^2]{E}{_1} \simeq H( \hcC[\lambda],\Delta_{0,1} ) \oplus \left( 
\bigoplus_{i=1}^N H( \hat{d}_i(\hcC_i) , \cD_i )
\right) .
\end{equation}

\subsection{} \label{cohC}

A vanishing result for the cohomology of $\hcC[\lambda]$ is  obtained by a simple degree counting argument. 
\begin{proposition} \label{vanC}
The cohomology $H^p_d(\hcC[\lambda], \Delta_{0,1})$ vanishes for all $(p,d)$, unless
\begin{equation}
d=0, \dots ,N, \quad 
p=d, \dots , d+N.
\end{equation}
\end{proposition}
\begin{proof}
The possible bi-degrees of the elements of $\hcC$ are precisely those excluded in the proposition. 
\end{proof}

\subsection{}
\label{cohdC}

We have the following vanishing result for the cohomology of $(\hat{d}_i (\hcC_i) , \cD_i)$.

\begin{proposition} \label{vandC}
The cohomology $H^p_d(\hat{d}_i(\hcC_i),  \cD_i )$ vanishes for all $(p,d)$, unless
\begin{equation}
d=2, \cdots, N+2, \quad 
p=d, \dots , d+N-1 .
\end{equation}
\end{proposition}
\begin{proof}
To prove this result let us introduce a third spectral sequence. For fixed $i$, let $\deg_{\theta^1_i}$ be the degree that assigns
degree one to  $\theta_i^1$ and degree zero to the remaining generators. Consider the decreasing filtration associated to the degree $\deg_{\theta_i^1} - \deg_\theta$. Let $\tensor*[^3]{E}{}$ be the associated spectral sequence. Let $\cD_{i,1}$ be the homogeneous component of $\cD_i$ with $\deg_{\theta_i^1} =1$, i.e., $\cD_{i,1} = \Psi \tilde{\cD}_{i,1} \Psi^{-1}$ with
\begin{equation}
\tilde{\cD}_{i,1} = \theta_i^1 \left[ \sum_j (u^i - u^j) \gamma_{ji} \theta_j^0 \frac{\partial }{\partial \theta_i^0} 
+\cE_i \right].
\end{equation}
The zero page $\tensor*[^3]{E}{_0}$ is given by $\hat{d}_i(\hcC_i)$ with differential $\cD_{i,1}$:
\begin{equation}
(\tensor*[^3]{E}{_0}, \tensor*[^3]{d}{_0} ) = ( \hat{d}_i(\hcC_i), \cD_{i,1}).
\end{equation}
To prove the proposition it is sufficient to prove the vanishing of the cohomology of this complex in the same degrees, which we will do in the next lemma. 
\end{proof}

\begin{lemma}
The cohomology $H^p_d(\hat{d}_i(\hcC_i), \cD_{i,1})$ vanishes for all $(p,d)$, unless
\begin{equation}
d=2, \cdots, N+2, \quad 
p=d, \dots , d+N-1 .
\end{equation}
\end{lemma}
\begin{proof}
As before let us work with the operator $\tilde{\cD}_{i,1}$. Let $\m$ be a monomial in the variables $u^{i,s}$, $\theta_i^{s+1}$ for $s\geq1$. For $g\in\hcC$, we have
\begin{equation}
\tilde{\cD}_{i,1} \left( g \hat{d}_i (\m) \right) = \theta_i^1 \left(
\sum_j (u^i - u^j) \gamma_{ji} \theta_j^0 \frac{\partial }{\partial \theta_i^0} g  + (w_i(g) + w_i(\m) -1) g 
\right) \hat{d}_i(\m),
\end{equation}
where $w_i$ is the weight defined in \cref{lemmaD01}.
Therefore $\tilde{\cD}_{i,1}$ leaves $\hcC \hat{d}_i(\m)$ invariant for each monomial $\m$. We will now prove that the cohomology of the subcomplex $\hcC \hat{d}_i(\m)$ vanishes for all monomials $\m$, except for the case $\m=u^{i,1}$, therefore the cohomology of $\hat{d}_i(\hcC_i)$ is just given by the cohomology of $\hcC \hat{d}_i(u^{i,1})$.
Notice that $\hat{d}_i(\m)$ is nonzero only for $w_i(\m) \geq \frac32$, and the case $w_i(\m)=\frac32$ corresponds to $\m=u^{i,1}$ and $\hat{d}_i(\m) = \theta_i^2$.

Let us split $\hcC = \hcC^i_0 \oplus \theta_i^0 \hcC^i_0$, where $\hcC^i_0$ is the subspace spanned by monomials that do not contain $\theta_i^0$. Given $g \in \hcC^i_0$ we have
\begin{equation}
\tilde{\cD}_{i,1} \left( g \hat{d}_i(\m) \right)= \theta_i^1(w_i(\m)-1)g \hat{d}_i(\m).
\end{equation}
Notice that the coefficient $w_i(\m)-1$ is non-vanishing, therefore $\tilde{\cD}_{i,1}$ is acyclic on the subcomplex $\hcC^i_0 \hat{d}_i(\m)$.
For $g \in \theta_i^0 \hcC^i_0$, the differential $\tilde{\cD}_{i,1}$ maps $g \hat{d}_i(\m)$ to $\theta_i^1(w_i(\m)-\frac32)g \hat{d}_i(\m) \in \theta^0_i \hcC^i_0 \hat{d}_i(\m)$ plus an element in $\hcC^i_0 \hat{d}_i(\m)$. 

It is well-known that when a complex $(C,d)$ contains an acyclic subcomplex $C'$, its cohomology is given by the cohomology of a subspace $C''$ complementary to $C'$ with differential given by the restriction and projection of $d$ to $C''$. 

In the present case this implies that the cohomology of $\hcC \hat{d}_i(\m)$ is equivalent to the cohomology of $\theta_i^1 \hcC^i_0$ with differential given by the operator of multiplication by the element $\theta_i^1(w_i(\m)-\frac32)$. Such complex is acyclic as long as $w_i(\m) \not= \frac32$. The only nontrivial case is when $\m = u^{i,1}$, and in such case the cohomology is given by
\begin{equation}
\theta_i^0 \hcC^i_0 \hat{d}_i(u^{i,1}) =\hcC^i_0 \theta_i^0 \theta_i^2 .
\end{equation}
Counting the degrees of the possible elements in this space we obtain the vanishing result above. 
\end{proof}

\subsection{}

From the previous two propositions it follows that $\tensor*[^2]{E}{_1}$ is zero if the bi-degree $(p,d)$ is not in one of the two specified ranges, i.e., in their union given in \cref{vanth}. Clearly the vanishing of $\tensor*[^2]{E}{_1}$ in certain degrees implies the vanishing of $\tensor*[^1]{E}{_2}$ and consequently of $H(\hcA[\lambda], D_\lambda)$ in the same degrees. This concludes the proof of \cref{vanth}.

\section{The cohomology of \texorpdfstring{$(\hat{d}_i(\hcC_i), \cD_i)$}{dhat(Ci,Di)}}
\label{subcomplexdCi}

In this section we extend the vanishing result of \cref{cohdC} to a computation of the full cohomology of the complex $(\hat{d}_i(\hcC_i), \cD_i)$.

First, we can represent the space $\hat{d}_i(\hcC_i)$ as a direct sum
\begin{equation}
\hat{d}_i(\hcC_i) =
\hcC_0^i \theta_i^2 \oplus \hcC^i_0 \theta_i^0 \theta_i^2 \oplus \hcC \otimes \hat{d}_i(V_i)
\end{equation}
where, as before in \cref{cohdC}, we denote by $\hcC^i_0$ the subspace of $\hcC$ spanned by monomials that do not contain $\theta_i^0$. We denote by $V_i$ the space of polynomials in $u^{i,\geq 1}$ , $\theta_i^{\geq2}$ of standard degree $\geq2$. 

\begin{lemma}
The differential $\cD_i$ leaves invariant the spaces $\hcC_0^i \theta_i^2$ and $\hcC \otimes \hat{d}_i(V_i)$, while
\begin{equation}
\cD_i ( \hcC_0^i \theta_i^0 \theta_i^2 ) \subset \hcC \theta_i^2 
= \hcC_0^i \theta_i^2 \oplus \hcC^i_0 \theta_i^0 \theta_i^2 .
\end{equation}
\end{lemma}
\begin{proof}
As before we can equivalently work with $\tilde{\cD}_i$. The statement is a simple check, noticing that $[\tilde{\cD}_i, \hat{d}_i ]_+ = - \theta_i^1 \hat{d}_i$.
\end{proof}

As we know from \cref{cohdC} the cohomology is a subquotient of $\hcC^i_0 \theta_i^0 \theta^2_i$. Therefore the subcomplexes $\hcC_0^i \theta_i^2$ and $\hat{d}_i(V_i)$ are acyclic and the cohomology is given by
\begin{equation}
H(\hat{d}_i(\hcC_i), \cD_i) = H( \hcC_0^i \theta_i^0\theta_i^2, \cD_i'),
\end{equation}
where $\cD_i'$ is the restriction and projection of $\cD_i$ to $ \hcC_0^i \theta_i^0\theta_i^2$. Explicitly $\cD_i' = \Psi \tilde{\cD}_i' \Psi^{-1}$ is given by removing the terms in $\tilde{\cD}_i$ that decrease the degree in $\theta_i^0$, which gives
\begin{equation}
\tilde{\cD}_i' =\sum_{k\not= i} \theta_k^1 \left[ (u^k - u^i) \left( \frac{\partial }{\partial u^k} + \sum_{j\not= i} \gamma_{jk} \theta_k^0 \frac{\partial }{\partial \theta_j^0}  \right)
+ \sum_j (u^i - u^j) \gamma_{jk} \theta_j^0 \frac{\partial }{\partial \theta_k^0} 
+\cE_k \right] .
\end{equation}
Notice that $\cE_i$ maps $\hcC^i_0 \theta_i^0 \theta^2_i$ to zero, since both $\theta_i^1$ and $\theta_i^0 \theta_i^2$ have degree $w_i$ equal to zero. We can now split $\hcC^i_0 \theta_i^0 \theta^2_i$ in the direct sum $\hcC^i_{0,1} \theta_i^0 \theta^2_i \oplus \hcC^i_{0,1} \theta_i^0 \theta_i^1 \theta^2_i$ where $\hcC^i_{0,1}$ is the subspace of $\hcC^i_0$ spanned by monomials that do not depend on $\theta_i^1$. Since $\tilde{\cD}_i'$ does not act on $ \theta_i^1 \theta^2_i$ or $\theta_i^0 \theta_i^1 \theta^2_i$, we can reduce our problem to computing the cohomology of the complex $(\hcC^i_{0,1}, \tilde{\cD}_i')$. Let us denote by $\hat{\delta}_k^i$ the coefficient of $\theta_k^1$ in $\tilde{\cD}_i'$, i.e., 
\begin{equation}
\tilde{\cD}_i' = \sum_{k\not=i} \theta_k^1 \hat{\delta}_k^i. 
\end{equation}

\begin{lemma}
The cohomology $H^p_d(\hcC^i_{0,1}, \tilde{\cD}_i')$ is nontrivial only in degrees $d=0$ and $p=0, \dots , N-1$. In degree $(d=0,p)$ it is isomorphic to $C^\infty (u^i) \otimes \bigwedge^p \R^{N-1}$ and is represented by an element 
\begin{equation}
F = \sum_{\substack{J \subseteq [n]\setminus \{ i \} \\ | J|=p} }
F^J(u^1, \dots , u^N) \theta^0_J
\in \bigcap_{k \not =i} \ker \hat{\delta}_k^i ,
\end{equation}
which depends on a single function of the variable $u^i$.
\end{lemma}

\begin{proof}

We represent the space of coefficients $\hcC^i_{0,1}$ as a direct sum $\bigoplus_{\ell,t=0}^{n-1} K^{\ell,t}$, where an element of $K^{\ell,t}$ can be written down as 
$$
\sum_{\substack{I\subset [n]\setminus \{i\} \\ |I|=\ell} }
 f^I \theta^1_I \cdot \sum_{\substack{J\subset [n]\setminus \{i\} \\ |J|=t }} \theta^0_JF_{I,J}(u^1,\dots,u^n).
$$
The action of $\mathcal{D}_i'$ can be described, in both cases, as a map $K^{\ell,t}\to K^{\ell+1,t}$ given by the following formula on the components of the corresponding vectors: $F_{I,J}\mapsto G_{S,T}$, where
$$
G_{S,T} = \sum_{s\in S} \frac{\d}{\d u^{s}}F_{I\setminus\{s\},T} +  (A_{s;t})^J_{T} F_{I\setminus\{s\},J},
$$
where the coefficients of the matrices $(A_{s;t})^J_T$ can easily be reconstructed from the formula for the operator $\tilde{\mathcal{D}}_i'$. So, this way we can describe each of the subcomplexes $K^{\bullet,t}\theta^0_i\theta^2_i$, $K^{\bullet,t}\theta^0_i\theta^1_i\theta^2_i$, $t=0,\dots,n-1$, as a tensor product of the de Rham complex of smooth functions in $n-1$ variable $u^k$, $k\not= i$, with a vector space whose basis is indexed by monomials of degree $t$ in $\theta^0_q$, $q\not=i$. The differential (the restriction of $\tilde{\mathcal{D}}_i'$ to this subcomplex) is equal to the de Rham differential $\sum_{p\not=i} \theta^1_p \frac{\d}{\d u^p}$ twisted by a linear map:
\begin{equation} \label{eq:deformed-de-Rham}
\sum_{p\not=i} \theta^1_p \cdot \left( \frac{\d}{\d u^p} + A_{p;t}\right).
\end{equation}
(the coefficients of $A_{p;t}$ depend on whether we consider the case of  $K^{\bullet,t}\theta^0_i\theta^2_i$ or $K^{\bullet,t}\theta^0_i\theta^1_i\theta^2_i$, but the shape of the differential is the same in both cases).

The cohomology of the differential~\eqref{eq:deformed-de-Rham} is isomorphic to the cohomology of the de Rham differential $\sum_{p\not=i} \theta^1_p  \frac{\d}{\d u^p}$. It is represented by the differential forms of order $0$, that is, it is non-trivial only for $\ell=0$, whose vector of coefficients $F_{\emptyset,J}$ solves the differential equations
\begin{equation}\label{eq:equation-for-coeff-F}
\frac{\partial F_{\emptyset,J}}{\partial u^p} + (A_{p;t})_J^T F_{\emptyset,T} = 0
\end{equation}
for $p\not=i$. The solution of this equation is uniquely determined by the restriction $F_{\emptyset,J}|_{u_p=0,\ p\not=i}$, that is, by a single function of $u^i$. So, finally, we obtain the statement of the lemma.
\end{proof}

Taking into account the action of $\Psi$ we obtain the cohomology the complex $(\hat{d}_i(\hcC_i), \mathcal{D}_i)$.

\begin{proposition} 
\label{lem13}
The cohomology of $(\hat{d}_i(\hcC_i), \mathcal{D}_i)$ is nontrivial only in degrees equal to $(p,d)=(2,2),\dotsc, (N+1,2)$ and $(p,d)=(3,3),\dots, (N+2,3)$. In the degrees $(2+t,2)$ and $(3+t,3)$ it is isomorphic to $C^\infty(u^i) \otimes \bigwedge^t\mathbb{R}^{N-1}$, $t=0, \dots , N-1$. 
More precisely, representatives of cohomology classes in degrees $(2+t,t)$ and $(3+t,3)$  are given respectively by elements of the form
\begin{equation}
F \cdot (f^i)^{t/2 +2}  \theta_i^0 \theta_i^2, \qquad G \cdot (f^i)^{t/2 +3}  \theta_i^0 \theta_i^1 \theta_i^2
\end{equation}
for $F$, $G$ representatives of $H^t_0(\hcC^i_{0,1} ,\tilde{\cD}_i')$ as given in the previous lemma. 
\end{proposition}

%

\section{The cohomology of \texorpdfstring{$(\hcC[\lambda], \Delta_{0,1})$}{C[l],D01} at \texorpdfstring{$p=d$}{p=d}}
\label{subcomplexpd}

In this section we extend the result of \cref{cohC} by computing the cohomology of the subcomplex of $(\hcC[\lambda], \Delta_{0,1})$ defined by setting $p=d$. 

From \cref{vanC} we already know that the complex \( (\hcC [\l ], \Delta_{0,1}) \) is non-trivial only for \( d \in \{ 0, \dotsc, n\}\) and \( p \in \{ d, \dotsc, d+n\} \). As usual, as the differential is of bidegree \( (p,d) = (1,1)\), it splits in subcomplexes of constant \( p-d \). Here we consider the case $p=d$.

\begin{proposition}
\label{pro14}
For $p=d$ the cohomology of the complex $(\hcC[\l], \Delta_{0,1})$ is given by
\begin{equation} 
H_p^p (\hcC [\l ], \Delta_{0,1} ) \simeq \begin{cases}
\R [\l ] & p = 0,\\
\bigoplus_{i=1}^N C^\infty (u^i) \th_i^1 & p= 1,\\
0 & \text{else.}
\end{cases}
\end{equation}
\end{proposition}
\begin{proof}
For $p=d$ the complex $\hcC[\lambda]$ is equal to 
\begin{equation} 
C^\infty (U) [ \theta^1_1 , \dots , \theta_N^1 ] .
\end{equation}
Let us compute the cohomology of $\tilde{\Delta}_{0,1}$. Because there is no dependence on $\theta^0_k$ and the degree $w_k$ of $\theta_k^1$ is zero, the differential simplifies to 
\begin{equation}
\tilde{\Delta}_{0,1} = \sum_i \delta_i ,
\quad 
\delta_i = (-\lambda+ u^i) \theta_i^1 \frac{\partial }{\partial u^i}. 
\end{equation}
We will let \( J \subseteq \{ 1, \dotsc, N\} \) denote a multi-index and write $\theta_J^1$ for the lexicographically ordered product $\prod_{j\in J} \theta_j^1$.
For each of the \( \th_1^1, \th_2^1, \dotsc, \th_N^1 \), we can define a degree \( \deg_{\th_i^1} - \deg_\th \), which again induces a decreasing filtration. The filtration associated to \( \th_i^1\) has \( \delta_i \) as differential on the zeroth page of the spectral sequence. Considering all these filtrations, we get the following picture:
\begin{equation}
\xymatrix@C=0pt{
& C^\infty (U) [\l ] \th_k^1 \ar[dd]^(0.3){\delta_i} \ar[rr]^{\delta_j} && C^\infty(U)[\l ] \th_j^1 \th_k^1\ar[dd]^{\delta_i} \\
C^\infty (U)[\l ] \ar[dd]_{\delta_i} \ar[rr]_(0.3){\delta_j} \ar[ru]^{\delta_k} && C^\infty (U)[\l ] \th_j^1 \ar[dd]^(0.3){\delta_i} \ar[ru]_{\delta_k} & \\
& C^\infty (U)[\l ] \th_i^1 \th_k^1 \ar[rr]_(0.35){\delta_j} && C^\infty (U)[\l ] \th_i^1 \th_j^1 \th_k^1 \\
C^\infty (U)[\l ] \th_i^1 \ar[rr]_{\delta_j} \ar[ru]_{\delta_k} && C^\infty (U)[\l ] \th_i^1 \th_j^1 \ar[ru]_{\delta_k}& 
}
\end{equation}
So the complex can be visualised as an \( N\)-dimensional hypercube with a term in every corner.

On the first page of the \( \th_1^1 \)-spectral sequence, the differential is  \( \sum_{j \neq 1} \delta_j \), and we can use the \( \th_2^1\)-filtration to get another spectral sequence. This procedure can be repeated inductively.

Consider an element in \( C^\infty (U)[\l ] \th_J^1\). Clearly it is in $\ker \delta_1$ if $J$ contains $1$ or if it does not depend on $u^1$:
\begin{equation}
\ker \delta_1 = \bigoplus_{J \ni 1} C^\infty (U)[\lambda]\theta_J^1 \oplus \bigoplus_{J \not\ni 1} C^\infty (u^2, \dots , u^N)[\lambda] \theta_J^1 ,
\end{equation}
where $C^\infty (u^2, \dots , u^N)$ denotes the functions in $C^\infty (U)$ which are constant in $u^1$.
On the other hand, we clearly have
\begin{equation}
\im \delta_1 = \bigoplus_{J \ni 1}(u^1-\l ) C^\infty (U) [\l ]  \th_J^1 ,
\end{equation}
therefore the first page of the spectral sequence is 
\begin{equation}
H(\hcC[\lambda], \delta_1 ) =  \bigoplus_{J \ni 1} \frac{C^\infty (U)[\lambda]}{(u^i-\lambda)C^\infty (U)[\lambda]} \theta_J^1\oplus \bigoplus_{J \not\ni 1} C^\infty (u^2, \dots , u^N)[\lambda] \theta_J^1.
\end{equation}
As these arguments do not depend on the \( \th_i^1 \) for \( i \neq 1 \) in any way, on the first page of the spectral sequence we can use the \( \th_2^1 \) filtration and use the same arguments to find the first page of its spectral sequence. Completing the induction, we get the following result for the \( \tilde{\Delta}_{0,1}\)-cohomology on \( \hcC [\l ]\):
\begin{equation}
\bigoplus_{ J \subseteq \{ 1, \dotsc, N\}} \frac{C^\infty ( \{ u^j\}_{j \in J} )[\l ]}{\sum_{j \in J} (u^j - \l )} \th_J^1
\end{equation}
where the sum in the denominator is an ideal sum. If \( |J| \geq 2\), this ideal sum contains the invertible element \( u^i - u^j = (u^i - \l ) - (u^j -\l )\) for \( i, j \in J\), so the cohomology is zero. The cohomology of $\tilde{\Delta}_{0,1}$ is therefore nontrivial only in degree zero, where it equals $\R[\lambda]$, and in degree one, where it is given by the sum $\bigoplus_{i=1}^N C^\infty (u^i) \th_i^1$. To find the cohomology of $\Delta_{0,1}$ we need to take into account the action of the operator $\Psi$. Hence the cohomology of $\Delta_{0,1}$ in degree one is $\bigoplus_{i=1}^N C^\infty (u^i) f^i(u) \th_i^1$.
The proposition is proved.
\end{proof}

\section{A vanishing result for \texorpdfstring{$\tensor*[^1]{E}{_2}$}{1E2} at \texorpdfstring{$(p,d)=(3,2)$}{(p,d)=(3,2)}}
\label{subcomplexpdpu}

We now go back to the first spectral sequence $\tensor*[^1]{E}{}$ associated with $\deg_u$ in \cref{degu} and prove a vanishing result for its second page.

\begin{proposition}
\label{prop32}
The cohomology of the complex $(H(\hcA[\lambda], \Delta_{-1}) , \Delta_0 )$ vanishes in degree $(p,d)=(3,2)$.
\end{proposition}
\begin{proof}
In \cref{sec:1E1} the vanishing result for $\tensor*[^1]{E}{_2}$ is proved by introducing a filtration in the degree $\deg_{\theta^1}$. In order to extend the vanishing to the case $(p,d)=(3,2)$, we split the differential $\Delta_0$ in a different way. Recall that the operator $\Delta_0$ is by definition the homogeneous component of $D_\lambda$ of degree $\deg_u$ equal to zero. It induces a differential on the first page $\tensor*[^1]{E}{_1}$ of the first spectral sequence, that is on the cohomology $H(\hcA[\lambda ], \Delta_{-1}) $ given by \cref{cohA}.

From \cref{lem13} we know that the cohomology of this complex is vanishing for $\deg_u$ positive. We can therefore limit our attention to the subcomplex with $\deg_u$ equal to zero  
\begin{equation}
\tensor*[^1]{E}{_1^0} =\hcC[\lambda]\oplus \bigoplus_{i=1}^N \frac{\hcC \llbracket \theta^{\geq 2}_i\rrbracket^\textnormal{nt} [\lambda]}{(\lambda-u^{i})\hcC \llbracket\theta^{\geq 2}_i\rrbracket^\textnormal{nt} [\lambda]}, 
\end{equation}
where the superscript in $\hcC\llbracket \theta^{\geq 2}_i\rrbracket^\textnormal{nt}$ indicates that every monomial should include at least one $\theta_i^{\geq 2}$. 

Let us denote by $\deg_{\theta^0}$ the degree that counts the number of $\theta^0_j$, $j=1, \dots , N$, and split $\Delta_0$ it its homogeneous components
\begin{equation}
\Delta_0 = \Delta_0^1 + \Delta_0^0,
\end{equation} 
where $\deg_{\theta^0} \Delta_0^k = k$. 

The decreasing filtration on $\tensor*[^1]{E}{_1^0}$ associated to the degree $\deg_\theta - \deg_{\theta^0}$ induces a spectral sequence $\tensor*[^4]{E}{}$, whose zero page is clearly $\tensor*[^1]{E}{_1^0}$, with differential $\tensor*[^4]{d}{_0} = \Delta_0^1$. The first page $\tensor*[^4]{E}{_1}$ is given by the cohomology of $(\tensor*[^1]{E}{_1^0},\Delta_0^1)$ which we now consider.

The form of $\Delta_0^1$ can be easily derived from the explicit expression of $\Delta_0$, see \cref{sec:D0}. When acting on $\tensor*[^1]{E}{_1^0}$ it simplifies to the following operator, which for simplicity we still denote $\Delta_0^1$:
\begin{equation}
\Delta_0^1 = \frac{1}{2} \sum_{i}  \tilde\theta_i^0 \sum_{s\geq1} \theta^{s+1}_i \frac{\d}{\d \theta^s_i},
\end{equation}
with 
\begin{equation}
\tilde\theta_i^0:= f^i \theta_i^0 + \sum_{j\not=i} ( u^i -u^j) \frac{f^j \d_j f^i}{f^i} \theta_j^0.
\end{equation}

We consider now the spectral sequence on $\tensor*[^1]{E}{_1^0}$ induced by the degree $\deg_{\theta^{\geq2}}$, which assigns degree one to all $\theta_i^s$ with $s\geq2$. Let $\Delta_0^1 = \Delta_0^{1,0} + \Delta_0^{1,1}$ where 
\begin{equation}
\Delta_0^{1,0} =  \frac{1}{2} \sum_{i}  \tilde\theta_i^0 \sum_{s\geq2} \theta^{s+1}_i \frac{\d}{\d \theta^s_i}, 
\quad 
\Delta_0^{1,1} =  \frac{1}{2} \sum_{i}  \tilde\theta_i^0 \theta^{2}_i \frac{\d}{\d \theta^1_i}, 
\end{equation}
are of degree $\deg_{\theta^{\geq2}} \Delta_0^{1,k} = k$.

We can rewrite our complex as 
\begin{equation}
\hcC[\lambda]\oplus \bigoplus_{i=1}^N \bigoplus_{k \geq1} \frac{\hcC\llbracket\theta^{\geq 2}_i \rrbracket^{(k)} [\lambda]}{(\lambda-u^{i})\hcC \llbracket \theta^{\geq 2}_i \rrbracket^{(k)} [\lambda]}, 
\end{equation}
where $\hcC \llbracket \theta^{\geq 2}_i\rrbracket^{(k)}$ denotes the homogeneous polynomials with $\deg_{\theta^{\geq2}}$ equal to $k$. 

Each of the summands is invariant under $\Delta_0^{1,0}$, so it forms a subcomplex whose cohomology we can compute independently. Notice that the differential vanishes on $\hcC[\lambda]$, while it acts like multiplication by $\tilde{\theta}^0_i$ on the $k=1$ subcomplex
$$
\hcC \theta^2_i \to \hcC \theta^3_i \to  \hcC \theta^4_i \to \cdots ,
$$
which is therefore acyclic except for the first term, where the cohomology is given by the kernel of the multiplication map, i.e., the ideal of $\tilde\theta_i^0$ in $\hcC$ multiplied by $\theta^2_i$. 

The first page of the spectral sequence is therefore given by 
\begin{equation} \label{firstpage}
\hcC[\lambda] 
\oplus \bigoplus_i \frac{\hcC \tilde{\theta}_i^0 \theta_i^2 [\lambda]}{(\lambda-u^i)\hcC \tilde{\theta}_i^0 \theta_i^2 [\lambda]}
\oplus \bigoplus_{k\geq 2}  \bigoplus_{i} 
H(\hcC \llbracket \theta^{\geq 2}_i\rrbracket^{(k)}, \Delta_0^{1,0}) .
\end{equation}

While it is not difficult to compute the cohomology groups appearing in the third summand, it can be easily seen that they give no contribution to  $\tensor*[^1]{E}{_2}$. Indeed, we know from \cref{lem13} that the cohomology with standard degree $d\geq4$ is a subquotient of $\hcC[\lambda]$, but the minimal degree of elements in the third summand above is $d = 5$. 

On this page the differential is induced by $\Delta_0^{1,1}$, which has  $\deg_{\theta^{\geq 2}}$ equal to one. When acting on the second summand $\hcC \tilde{\theta}_i^0 \theta_i^2$  it vanishes, since it produces a mixed term $\theta^2_i \theta^2_j$ which cannot be in $\hcC \llbracket \theta^{\geq 2}_i\rrbracket^{(k)}$ for $k\geq 2$. Therefore the cohomology of the first two summands is determined by the kernel and the image of the map
\begin{equation}
\Delta_0^{1,1}:
\hcC[\lambda] 
\to \bigoplus_i \frac{\hcC \tilde{\theta}_i^0 \theta_i^2 [\lambda]}{(\lambda-u^i)\hcC \tilde{\theta}_i^0 \theta_i^2 [\lambda]} .
\end{equation}

The image can be computed in the following way: first of all, it is clear that an element in the image is a linear combination of $\theta^2_i$, $i=1,\dots,N$, where the coefficient of each $\theta^2_i$ does not depend on $\theta^1_i$ and is in the ideal generated by $\tilde\theta_i^0$ in $\hcC$. Therefore the image is a subspace of 
\begin{equation}
\label{imageA011}
\bigoplus_{i=1}^N \frac{ \hcC_1^i \tilde{\theta}_i^0 \theta_i^2 [\lambda]}{(\lambda-u^i)},
\end{equation}
where $\hcC_1^i$ is the subspace of $\hcC$ generated by monomials that do not depend on $\theta_i^1$.
Second, it is sufficient to consider the fact that the image of the ideal
$\prod_{j\not=i} (-\lambda+u^j) \hcC[\lambda]$ under $\Delta_{0}^{1,1}$ is
\begin{equation}
\frac{\hcC_1^i \tilde{\theta}_i^0 \theta_i^2 [\lambda]}{(\lambda-u^i)}
\end{equation}
to conclude that the image of $\Delta_0^{1,1}$ is the whole space~\eqref{imageA011}.

So, the cohomology of  $\Delta_0^{1,1}$ on the second term in~\eqref{firstpage} is 
\begin{equation}
\bigoplus_{i=1}^N \frac{ \hcC_1^i \tilde{\theta}_i^0 \theta_i^1 \theta_i^2 [\lambda]}{(\lambda-u^i)}.
\end{equation}
In particular, we see that it cannot give any contribution to the cohomology of degree $(p,d)=(3,2)$. 

The second page of the spectral sequence associated to $\deg_{\theta^{\geq2}}$ is 
\begin{equation}
\label{cohA01}
\ker \Delta_0^{1,1}|_{\hcC[\lambda]} \oplus 
\bigoplus_{i} \frac{ \hcC_1^i \tilde{\theta}_i^0 \theta_i^1 \theta_i^2 [\lambda]}{(\lambda-u^i)} \oplus 
\bigoplus_{k\geq 2}  \bigoplus_{i} 
H(H(\hcC \llbracket \theta^{\geq 2}_i \rrbracket^{(k)}, \Delta_0^{1,0}), \Delta_0^{1,1}) ,
\end{equation}
where, as discussed before, the third summand does not give any contribution to $\tensor*[^1]{E}{_2}$, and can therefore be ignored here. Since $\Delta_0^1$ vanishes on this page, \cref{cohA01} gives the cohomology of $(\tensor*[^1]{E}{_1^0},  \Delta_0^1)$ which coincides with the first page $\tensor[^4]{E}{_1}$ of the spectral sequence $\tensor*[^4]{E}{}$.

The differential $\,^4d_1$ on $\,^4E_1$ is the one induced by $\Delta_0^0$, the degree $\deg_{\theta^0}$ zero part of $\Delta_0$. The three summands in \cref{cohA01} are invariant under the action of the differential $\Delta_0^0$, which in particular vanishes on the second term. To see this, observe that since the standard degree of the second term is $d=3$ and that of the third term is $d\geq5$, there can be no terms mapped between these two spaces by $\Delta_0^0$, nor from the second space to itself. The third term cannot map to the first one, since $\Delta_0^0$ cannot remove more than one $\theta^{\geq2}$. 

The operator $\Delta_0^0$ has to increase the standard degree and the $\theta$-degree by one,  while keeping $\deg_{\theta^0}$ unchanged. This can only be achieved on $\hcC[\lambda]$ by increasing $\deg_{\theta^1}$ by one, implying $\Delta_0^0 = \Delta_{0,1}$, which is given in \cref{lemmaD01}. Explicitly:
\begin{align}
\Delta_{0}^0 &= (u^i - \l )f^i   \th_i^1 \big( \frac{\d}{\d u^i} - (\d_i \log f^k) \th_k^1 \frac{\d}{\d \th_k^1}  - \frac{1}{2} (\d_i \log f^k) \th_k^0 \frac{\d}{\d \th_k^0} \big) \\
 & \qquad - \frac{1}{2} (u^j-\l ) \d_i f^j \th_j^1 \th_j^0 \frac{\d}{\d \th_i^0} + \frac{1}{2} f^i \tilde{\theta_i^0} \th_i^1 \frac{\d}{\d \th_i^0} + (u^i - \l ) f^j \frac{\d_j f^i}{f^i} \th_j^0 \th_i^1 \frac{\d}{\d \th_i^0}.
\end{align}
From this formula it is easy to see that $\Delta_0^0$ maps $\hcC[\lambda]$ to itself. Finally, from the formula for $\Delta_0$, we easily see that there are no terms that remove the dependence on $\theta^2_i$ in the second summand in \cref{cohA01}, therefore such summand cannot map to the first. 

To get the second page $\,^4E_2$ we therefore need to compute the cohomology of the differential $\Delta_0^0$ on $\ker \Delta_0^{1,1}|_{\hcC[\lambda]} $. The discussion so far was for general bidegrees $(p,d)$. However to be able to say something more we need to restrict to the subcomplex $p=d+1$. 

We see that an element proportional to $\theta^1_i$ is in the kernel of $\Delta_0^{1,1}$ if and only if it is also proportional either to $(-\lambda+u^i)$ or to $\tilde\theta^0_i$. Therefore, it can be represented as a sum over all subsets $I\subset \{1,\dots,n\}$, $|I|=t$, of the elements of the form 
$$
 \sum_{j=1}^n F^j (u,\lambda) \theta^0_j \cdot \prod_{i\in I} (-\lambda+u^i) \theta^1_i + \sum_{i\in I} G^i(u) \tilde \theta^0_i \theta^1_i \cdot \prod_{\substack{j\in I \\ j\not=i}} (-\lambda+u^j) \theta^1_j .
$$
This representation naturally splits the kernel of $\Delta_0^{1,1}$ into two summands, let us call them $F$ and $G$. 

Observe that the splitting of the $p=d+1$ part of the kernel of  $\Delta_0^{1,1}$ on $\hcC[\lambda]$ into the direct sum $F\oplus G$ defines a filtration for the operator $\Delta_0^0 = \Delta_{0,1}$.
We can see this by using the base change \( \Psi \). First, define
\begin{equation}
\bar{\th}_i^0 \coloneqq \Psi^{-1} \tilde{\th}_i^0 = \th_i^0 +2 (u^j - u^i ) \gamma_{ji} \th_j^0
\end{equation}
From the formula above for $\Delta_0^0$ we have that we can write \( \Delta_0^0 = \Psi \bar{\Delta} \Psi^{-1} \), for
\begin{align}
\bar{\Delta} &= (u^i - \l ) \th_i^1 \frac{\d}{\d u^i} + (u^i - \l ) \gamma_{ji} \th_i^1 \th_i^0 \frac{\d}{\d \th_j^0} - (u^i -\l ) \gamma_{ji} \th_i^1 \th_j^0 \frac{\d}{\d \th_i^0} + \frac{1}{2} \bar{\th}_i^0 \th_i^1 \frac{\d}{\d \th_i^0}
\end{align}
The first three terms preserve \( F = \Psi^{-1}F\), while the last sends \( F \) to \( \bar{G} \coloneqq \Psi^{-1}G\). Moreover, the entire operator preserves \( \bar{G} \). Furthermore, the parts \( F \rightarrow F \) and \( \bar{G} \rightarrow \bar{G} \) form deformed de Rham differentials \( d + A \). Therefore, the only possible cohomology is in the lowest degree in $\theta^1_\bullet$, which is zero for $F$ and $1$ for $G$. So, only nontrivial cohomology in the case $p=d+1$ is possible in the degree $(t+1,t)=(1,0)$ and $(t+1,t)=(2,1)$. This implies the the cohomology of degree $(3,2)$ is equal to zero.
\end{proof}


\begin{remark}\label{DeltapresG}
Note that it is not directly clear from the definitions that  $\bar{\Delta} \bar{G} \subset \bar{G}$. However, we know that \( \bar{\Delta} \) must preserve the kernel of $\Delta_0^{1,1}$ twisted by \( \Psi \), which is \( F \oplus \bar{G} \). Moreover, looking at the \( \lambda \)-degree, we see that for elements of \( \bar{G}\) it is one less \( \deg_{\th^1} \) while for elements of \( F\) it is at least \( \deg_{\th^1} \). As \( \deg_{\th^1} \bar{\Delta} =1 \), and none of its terms increase the \( \lambda\)-degree by more than \( 1\), this proves that \( \bar{\Delta} \) cannot map \( \bar{G} \) outside of \( \bar{G} \). A more direct proof requires Ferapontov's flatness equations for $f^i$\cite{Ferapontov01}. 
We give this calculation in \cref{sec:D0}.\par
\end{remark}

\begin{remark}
In the proof, we restricted to  \( p = d+1\). In order to extend the argument, one would have to show that the transformation \( \th_i^0 \mapsto \bar{\th}_i^0 \) is invertible. This would allow for a splitting similar to the splitting in \( F\) and \( G\) here.
\end{remark}

\section{Proofs of the main theorems}
\label{cohomologyA}
In this section we collect all results from the rest of the paper to compute the cohomology of the complex $(\hcA[\lambda], D_\lambda)$, proving \cref{resultp=d,extravanishing}.

\begin{proof}[Proof of \cref{resultp=d}]
As observed in \cref{D01}, the first page $\tensor*[^2]{E}{_1}$ is given by the direct sum~\eqref{2E1}. From \cref{lem13,pro14} we get
\begin{equation}
(\tensor*[^2]{E}{_1})^p_p \cong \begin{cases}
\R [\l ] & p = 0,\\
\bigoplus_{i=1}^N C^\infty (u^i) \th_i^1 & p = 1,\\
\bigoplus_{i=1}^N C^\infty (u^i) \th_i^0 \th_i^2 & p =2,\\
\bigoplus_{i=1}^N C^\infty (u^i) \th_i^0 \th_i^1 \th_i^2 & p = 3,\\
0 & \text{else}.
\end{cases}
\end{equation}

On this first page, the differential $\tensor*[^2]{d}{_1}$ must lower the spectral sequence degree $\deg_{\theta^1} - \deg_{\theta}$ by one, in other words, since the differential must still be of bidegree $(1,1)$, it must leave the degree $\deg_{\theta^1}$ unchanged, which is impossible on this subcomplex. Hence, the differential \( \tensor*[^2]{d}{_1} \) is equal to zero, and \( (\tensor*[^2]{E}{_2})_p^p \cong (\tensor*[^2]{E}{_1})_p^p \).

On the second page, the differential $\tensor*[^2]{d}{_2}$ must lower the spectral sequence degree by two, i.e., it must be of degree $\deg_{\theta^1}$ equal to $-1$. Therefore, on this subcomplex the differential can only be non-trivial between \( p =1\) and \( p =2\). Looking back at the formula for \( \Delta_0 \), one can easily identify the terms of degree $\deg_{\theta_1} = -1$, which give
\begin{align}
\Delta_{0,-1}= \sum_{i} \frac{1}{2}  \Big[ \sum_j(u^j-\l ) \big( \d_i f^j \th_j^0 \th_j^2 +f^j \frac{\d_j f^i}{f^i} (\th_i^0 \th_j^2 - \th_j^0 \th_i^2) \big)+  f^i \th_i^0 \th_i^2 \Big] \frac{\d}{\d \th_i^1} .
\end{align}

$\Delta_{0,-1}$ induces an operator on $H(\hcA[\lambda], \Delta_{-1})$. Since we are interested only in the differential at degree $p=1$, we need to consider just the action of such operator on $\hcC[\lambda]$, which is, taking into account the identification~\eqref{ide}
\begin{equation}
\sum_i \frac12 \Big[ (u^i-u^j) \frac{f^j}{f^i} \partial_j f^i \theta^0_j \theta^2_i + f^i \theta^0_i \theta^2_i \Big] \frac{\d}{\d \th_i^1} .
\end{equation}
The image of $\hcC[\lambda]$ through this operator is thus in $\bigoplus_i H(\hat{d}_i(\hcC_i) , \cD_i)$, where the first term, being in $\hcC^i_{0,1} \theta_i^2$ vanishes. 
Hence, the only surviving term is \( \frac{1}{2} f^i \th_i^0 \th_i^2 \frac{\d}{\d \th_i^1}\), which gives an isomorphism \( \tensor*[^2]{d}{_2} : (\tensor*[^2]{E}{_2})_1^1 \rightarrow (\tensor*[^2]{E}{_2})_2^2 \). 

The differential is therefore zero on $(\tensor*[^2]{E}{_2})^p_p$ for $p\not=1$ and an isomorphism for $p=1$, so \( (\tensor*[^2]{E}{_3})_p^p \) is zero unless $p=0$ or $p=3$, when it is equal to $(\tensor*[^2]{E}{_2})_p^p$. This spectral sequence has no other non-trivial differentials, so \( (\tensor*[^2]{E}{_\infty} )_p^p \) has the same form. As \( \tensor*[^2]{E}{} \Longrightarrow \,\tensor*[^1]{E}{_2} \), we get that \( (\tensor*[^1]{E}{_2})_p^p \) is of this form as well. Because all differentials must have \( (p,d)\)-bidegree \( (1,1)\), there can be no higher non-trivial differentials on this part of the first spectral sequence. Now, \( \tensor*[^1]{E}{} \Longrightarrow H(\hcA [\l ], D_\l )\), yielding the result.

\end{proof}

\begin{proof}[Proof of \cref{extravanishing}]
We take \cref{vanth} as a starting point. Then the extra vanishing at degrees $d=N, N+1$ follows from \cref{lem13}, and the vanishing at $(3,2)$ follows from \cref{prop32}.
\end{proof}

\appendix
\section{Formula for and calculations with \texorpdfstring{$\Delta_0$}{D0}}
\label{sec:D0}
We recall from~\cite{cps15} the formula for the degree $\deg_u$ zero part of the operator $D_\lambda$.
\begin{align} \label{eq:Delta0}
\Delta_0 & = 
(-\lambda+u^i)f^i\theta_i^{1}\frac{\d}{\d u^i} 
\\ \notag &
+  \! \!
\sum_{\substack{
		s=a+b \\
		s, a \geq 1;
		b\geq 0
	}} \! \! (-\lambda+u^i) \binom {s}{b} \partial_j f^i u^{j,a} \theta_i^{1+b} \frac{\d}{\d u^{i,s}}
	+  \! \! \sum_{\substack{
			s=a+b \\
			s, a \geq 1;
			b\geq 0
		}}  \! \! \binom {s}{b} f^i u^{i,a} \theta_i^{1+b} \frac{\d}{\d u^{i,s}}
		\\ \notag
		& +  \! \! \frac{1}{2}\sum_{\substack{
				s=a+b \\
				s \geq 1;
				a,b\geq 0
			}}   \! \! (-\lambda+u^i) \binom{s}{b}
			\d_jf^i u^{j,1+a} \theta_i^b
			\frac{\d}{\d u^{i,s}}
			+\frac{1}{2}  \! \! \sum_{\substack{
					s=a+b \\
					s \geq 1;
					a,b\geq 0
				}}  \! \! \binom{s}{b}
				f^i u^{i,1+a} \theta_i^b
				\frac{\d}{\d u^{i,s}}
				\\ \notag
				& +\frac{1}{2}  \! \! \sum_{\substack{
						s=a+b \\
						s \geq 1;
						a,b\geq 0
					}}  \! \! (-\lambda+u^i) \binom{s}{b}  
					f^i\frac{\d_i f^j}{f^j} u^{j,1+a}\theta_j^b 
					\frac{\d}{\d u^{i,s}}
					+\frac{1}{2}  \! \! \sum_{\substack{
							s=a+b \\
							s \geq 1;
							a,b\geq 0
						}}  \! \! \binom{s}{b}  
						f^i u^{i,1+a}\theta_i^b 
						\frac{\d}{\d u^{i,s}}
						\\ \notag
						& -\frac{1}{2}  \! \! \sum_{\substack{
								s=a+b \\
								s \geq 1;
								a,b\geq 0
							}}  \! \!
							(-\lambda+u^j) \binom{s}{b} 
							f^j\frac{\d_j f^i}{f^i} u^{i,1+a}\theta_j^b
							\frac{\d}{\d u^{i,s}}
							-\frac{1}{2}  \! \! \sum_{\substack{
									s=a+b \\
									s \geq 1;
									a,b\geq 0
								}}  \! \!
								\binom{s}{b} 
								f^i u^{i,1+a}\theta_i^b
								\frac{\d}{\d u^{i,s}}
								\\ \notag
								& + \frac{1}{2} \sum_{\substack{
										s=a+b \\
										s,a,b\geq 0
									}}  
									(-\lambda+u^j)\binom{s}{b}
									\d_if^j  \theta_j^a \theta_j^{1+b} 
									\frac{\d}{\d \theta_i^s}
									+ \frac{1}{2} \sum_{\substack{
											s=a+b \\
											s,a,b\geq 0
										}}  
										\binom{s}{b}
										f^i  \theta_i^a \theta_i^{1+b} 
										\frac{\d}{\d \theta_i^s}
										\\ \notag
										&
										+ \frac{1}{2} \sum_{\substack{
												s=a+b \\
												s,a,b\geq 0
											}}  
											(-\lambda+u^j)\binom{s}{b}
											f^j\frac{\d_j f^i}{f^i} \theta_i^a \theta_j^{1+b}   \frac{\d}{\d \theta_i^s}
											+\frac{1}{2} \sum_{\substack{
													s=a+b \\
													s,a,b\geq 0
												}}  
												\binom{s}{b}
												f^i \theta_i^a \theta_i^{1+b}   \frac{\d}{\d \theta_i^s}
												\\ \notag
												&
												- 
												\frac{1}{2} \sum_{\substack{
														s=a+b \\
														s,a,b\geq 0
													}}  
													(-\lambda+u^j)\binom{s}{b}
													f^j\frac{\d_j f^i}{f^i} \theta_j^a \theta_i^{1+b}
													\frac{\d}{\d \theta_i^s}
													- 
													\frac{1}{2} \sum_{\substack{
															s=a+b \\
															s,a,b\geq 0
														}}  
														\binom{s}{b}
														f^i \theta_i^a \theta_i^{1+b}
														\frac{\d}{\d \theta_i^s} .														
\end{align}

The direct proof that \( \bar{\Delta} \bar{G} \subset \bar{G} \) in \cref{prop32} is given below. Recall that its validity is deduced more abstractly in \cref{DeltapresG} as well.\par
\begin{lemma}
The operator \( \bar{\Delta} \) preserves \( \bar{G} \), where
\begin{align}
\bar{\Delta} &= (u^i - \l ) \th_i^1 \frac{\d}{\d u^i} + (u^i - \l ) \gamma_{ji} \th_i^1 \th_i^0 \frac{\d}{\d \th_j^0} - (u^i -\l ) \gamma_{ji} \th_i^1 \th_j^0 \frac{\d}{\d \th_i^0} + \frac{1}{2} \bar{\th}_i^0 \th_i^1 \frac{\d}{\d \th_i^0}
\end{align}
and 
\begin{align}
\bar{G} = \bigoplus_{i=1}^N C^\infty (U) \Big[ \big\{ (u^j -\l ) \th_j^1\big\}_{j=1}^N \Big] \bar{\th}_i^0 \th_i^1
\end{align}
\end{lemma}
\begin{proof}
When calculating the action of \( \bar{\Delta} \) on an element of the form \( G(u) \bar{\th}^0_i \th^1_i \in \bar{G}\), we get the following (where \( i\) is a fixed index, and \( k\), \( l\), and \( m\) are summed over)
\begin{align}
\bar{\Delta}G(u) \bar{\th}^0_i \th^1_i &= \frac{\d}{\d u^k} \big( G (\th_i^0 + 2(u^l - u^i ) \gamma_{li} \th_l^0 )\big) \th_i^1  (u^k-\l )\th_k^1 \\
 & \quad + G \gamma_{mk} \th_k^0 \frac{\d}{\d \th_m^0} \big( \th_i^0 + 2 (u^l - u^i )\gamma_{li} \th_l^0\big) \th_i^1  (u^k-\l )\th_k^1 \\ 
&\quad - G \gamma_{lk} \th_l^0   \frac{\d}{\d \th_k^0} (\th_i^0 + 2(u^m - u^i ) \gamma_{mi} \th_l^0 ) \th^1_i (u^k -\l ) \th_k^1\\
& \quad +\frac{1}{2}  G \frac{\d}{\d \th_k^0}\big( \th_i^0 + 2(u^l - u^i ) \gamma_{li} \th_l^0 \big) \th_i^1 \bar{\th}_k^0 \th_k^1 \displaybreak[0] \\ 
&= \frac{\d G}{\d u^k} \bar{\th}_i^0 \th_i^1  (u^k - \l ) \th_k^1 + 2 G \gamma_{ki} \th_k^0 \th_i^1 (u^k - \l )\th_k^1\\
&\quad + 2 G (u^l - u^i) \d_k \gamma_{li} \th_l^0 \th_i^1(u^k-\l )\th_k^1 + G\gamma_{ik} \th_k^0 \th_i^1 (u^k -\l )\th_k^1 \\
&\quad + 2G (u^l-u^i) \gamma_{lk} \gamma_{li} \th_k^0 \th_i^1 (u^k -\l )\th_k^1 \\
&\quad- 2 G (u^k-u^i) \gamma_{lk} \gamma_{ki} \th_l^0 \th_i^1 (u^k-\l )\th_k^1 + G (u^k- u^i) \gamma_{ki} \th_i^1 \bar{\th}_k^0 \th_k^1
\end{align}
Using \cref{eq:dgammaijk} for the third term if \( i,k,l\) distinct, that part of the third term adds up to the sixth term.
\begin{align}
\bar{\Delta}G(u) \bar{\th}^0_i \th^1_i &= \frac{\d G}{\d u^k} \bar{\th}_i^0 \th_i^1  (u^k - \l ) \th_k^1 + 2 G \gamma_{ki} \th_k^0 \th_i^1 (u^k - \l )\th_k^1\\
&\quad + 2 G (u^k - u^i) \d_k \gamma_{ki} \th_k^0 \th_i^1(u^k-\l )\th_k^1 + G\gamma_{ik} \th_k^0 \th_i^1 (u^k -\l )\th_k^1 \\
&\quad + 2 G (u^l-u^i) \gamma_{lk} \gamma_{li} \th_k^0 \th_i^1 (u^k -\l )\th_k^1 + G ( u^i-\l ) \gamma_{ki} \bar{\th}_k^0 \th_i^1 \th_k^1\\
&\quad + 2 G (u^l-u^k) \gamma_{lk} \gamma_{ki} \th_l^0 \th_i^1 (u^k-\l )\th_k^1 - G \gamma_{ki} \bar{\th}_k^0 \th_i^1 (u^k- \l ) \th_k^1
\end{align}
By the definition of \( \bar{\th}_k^0 \), the last two terms drop out against half of the second term. So we get
\begin{align}
\bar{\Delta}G(u) \bar{\th}^0_i \th^1_i &= \frac{\d G}{\d u^k} \bar{\th}_i^0 \th_i^1  (u^k - \l ) \th_k^1 + G (\gamma_{ik} + \gamma_{ki}) \th_k^0 \th_i^1 (u^k - \l )\th_k^1\\
&\quad + 2 G (u^k - u^i) \d_k \gamma_{ki} \th_k^0 \th_i^1(u^k-\l )\th_k^1 \\
&\quad + 2 G (u^l-u^i) \gamma_{lk} \gamma_{li} \th_k^0 \th_i^1 (u^k -\l )\th_k^1 + G ( u^i-\l ) \gamma_{ki} \bar{\th}_k^0 \th_i^1 \th_k^1
\end{align}
By \cref{udgamma}, we get
\begin{align}
\bar{\Delta}G(u) \bar{\th}^0_i \th^1_i &= \frac{\d G}{\d u^k} \bar{\th}_i^0 \th_i^1  (u^k - \l ) \th_k^1 - G u^i\d_i \gamma_{ik} \th_k^0 \th_i^1 (u^k - \l )\th_k^1\\
&\quad - 2 G u^i \d_k \gamma_{ki} \th_k^0 \th_i^1(u^k-\l )\th_k^1 \\
&\quad - 2 G u^i \gamma_{lk} \gamma_{li} \th_k^0 \th_i^1 (u^k -\l )\th_k^1 - G  \gamma_{ki} ( u^i-\l ) \th_i^1 \bar{\th}_k^0 \th_k^1
\end{align}
Applying \cref{dgammaij} gives
\begin{align}
\bar{\Delta}G(u) \bar{\th}^0_i \th^1_i &= \frac{\d G}{\d u^k} \bar{\th}_i^0 \th_i^1  (u^k - \l ) \th_k^1 - G  \gamma_{ki} ( u^i-\l ) \th_i^1 \bar{\th}_k^0 \th_k^1
\end{align}
Multiplying with a factor \( \prod_{j \in I} (u^j - \l ) \th_j^1 \) does not change the calculation, so we can extend this calculation to all of \( \bar{G}\), showing that \( \bar{\Delta} \) does indeed preserve this space.
\end{proof}

\printbibliography

\end{document}